\documentclass[a4paper,11pt]{article}
\usepackage{amsmath}
\usepackage{amsfonts}
\usepackage[latin9]{inputenc}
\usepackage[english]{varioref}
\usepackage{dcolumn}
\usepackage[height=22cm , width = 16cm , top = 2.5cm , left = 3cm, a4paper]{geometry}
\usepackage[a4paper]{geometry}
\usepackage[final]{graphicx}
\usepackage{epsfig}
\usepackage{pstricks}
\usepackage{psfrag}
\usepackage{rotating}
\usepackage{booktabs}
\usepackage{delarray}
\usepackage{amssymb}
\usepackage{rotating}
\usepackage{subfigure}
\usepackage{layout}
\usepackage{amsthm}
\usepackage{hyperref}
\usepackage{xcolor}
\usepackage{setspace}
{\newcommand{\dd}{\,d}
{
{
{\newcommand{\pk}{,}

\newtheorem{thm}{Theorem}[section]

\newtheorem{lem}[thm]{Lemma}
\newtheorem{propos}[thm]{Proposition}
\newtheorem{rem}[thm]{Remark}

\parindent 0pt
\newcommand{\enter}{\bigskip}

\begin{document}
\author{Sanjiv Kumar Bariwal${}^1$, Prasanta Kumar Barik${}^2$, Ankik Kumar Giri${}^3$, Rajesh Kumar${}^1$\footnote{{\it{${}$ Email address:}} rajesh.kumar@pilani.bits-pilani.ac.in}\\
\footnotesize ${}^1$Department of Mathematics, Birla Institute of Technology and Science, Pilani,\\ \small{ Pilani-333031, Rajasthan, India}\\
\footnotesize ${}^2${School of Mathematics, Indian Institute of Science Education and Research Thiruvananthapuram,} \\ \small{Thiruvananthapuram-695551, Kerala, India}\\
\footnotesize ${}^3$Department of Mathematics, Indian Institute of Technology Roorkee,\\ \small{ Roorkee-247667, Uttarakhand, India}\\
}

\title {Numerical analysis for coagulation-fragmentation equations with singular rates}

\maketitle

\hrule \vskip 8pt

\begin{quote}
{\small {\em\bf Abstract}}: This article deals with the convergence of finite volume scheme (FVS) for solving coagulation and multiple fragmentation equations having locally bounded coagulation kernel but singularity near the origin due to fragmentation rates. Thanks to the Dunford-Pettis and De La Vall$\acute{e}$e-Poussin theorems which allow us to have the convergence of numerically truncated solution towards a weak solution of the continuous model using a weak $L^1$ compactness argument. A suitable stable condition on time step is taken to achieve the result. Furthermore, when kernels are in $W_{loc}^{1,\infty}$ space, first order error approximation is demonstrated for a uniform mesh. It is numerically validated by attempting several test problems.
\end{quote}

\section{Introduction}
The equations, in this article illustrate the binary coagulation and multiple fragmentation processes in which two particles can coalesce to form a large particle and one particle breaks into arbitrary small fragments. Such processes are seen in a variety of physical system, for instance in aerosol science, astrophysics, colloidal chemistry, polymer science, meteorology (merging of drops in atmospheric clouds), crystallization, see \cite{lee2001survey, gokhale2009disintegration} and references therein. The recognition of each particle is defined by its size, in particular, mass or volume which could be either a positive integer (discrete system) or a positive real number (continuous system). The coagulation-fragmentation model describes the time evolution of the particle size distribution.\enter

If $c(t,x) \ge 0$ is the particle number density function of having particles of volume $x \in \mathbb{R}_{>0}:=]0,\infty[$ and at time $t\in [0, \infty [$ in a homogeneous physical system, then the mathematical formulation due to simultaneous aggregation and breakage processes is governed by the following well known population balance equations (PBEs),
\begin{align}\label{maineq}
\frac{\partial c(t,x)}{\partial t}\ \ =\ \ &\frac{1}{2}\int_0^x K(y,x-y)c(t,y)c(t,x-y)dy-\int_0^\infty K(x,y)c(t,x)c(t,y)dy \nonumber\\
			 & +  \int_x^\infty B(x,y)S(y)c(t,y)dy-S(x)c(t,x)
\end{align}
with the given initial data
\begin{align}
			c(0,x)\ \ = \ \ c^{in}(x) \geq 0, \ \ \ x \in \mathbb{R}_{>0}.
\end{align}
Here $K(x,y)$ expresses the rate at which particles of sizes $x$ and $y$ collide and gives birth to a particle of size $x+y$. It is assumed to be non-negative ($K(x,y)\geq 0$) and symmetric, i.e. $K(x,y)=K(y,x)$. The fragmentation kernels are defined by the breakage function $B(x,y)$ and the selection rate $S(y)$ as follows:
\begin{itemize}
\item $B(x,y)$ indicates the rate at which particles of size $y$ breaks out, and produces  particles of sizes $x$. Also $ B(x,y)\neq 0$ only for $x<y$.
\item $S(y)$ denotes the rate at which particles of sizes $y$ are selected to break.
\end{itemize}
The breakage function holds the following properties:
\begin{align*}
		\int_{0}^{y} B(x, y)dx =\zeta(y), \hspace{0.4cm} \int_{0}^{y} xB(x, y)dx=y	
\end{align*}
for $\zeta(y)$ being the total number of daughter fragments due to splitting of a particle of size $y$ and second relation is the necessary condition for the mass conservation. Assume that $\sup_{y} \zeta(y) = \eta$, for every $y \in \mathbb{R}_{>0}$. The first and third integrals in the equation (\ref{maineq}) give the production of a particle of size $x$ while the second and fourth integrals involve in the disappearance of a particle of size $x$ due to the aggregation and fragmentation processes, respectively. \\

Besides solution, it should be focused here that the integral properties like moments are also of useful. The $jth$ moment of the particle size distribution is defined as
\begin{align}\label{I:exactmoment}
	\mu_j(t) := \int_{0}^{\infty} x^j c(t,x) dx.
\end{align}
For a particular value of $j=0$, zeroth moment is proportional to the total number of the particles while for $j=1$, it describes the total mass of the particles in the system. One can easily show that the zeroth moment decreases by coagulation and increases by breakage processes while the total mass stays constant for some specific rates. For the total mass conservation, the integral equality
\begin{align*}
	\int_0^\infty x c(t,x) \dd x = \int_0^\infty x c^{in}(x) \dd x, \ \
	t\geq 0 \pk
\end{align*}
holds. \enter

Several researchers have studied the existence of a solution of PBEs, see in \cite{melzak1957scalar, mcleod1962infinite, ball1990discrete}.
 In continuing, Stewart \cite{stewart1990coagulation, stewart1991density,stewart1989global} as well as Lauren{\c{c}}ot \cite{laurenccot2000class, laurenccot2002discrete}, dealt with continuous PBEs using compactness arguments in the space of integrable functions.
  However, a lot of work has been developed related to the existence of weak solutions for coagulation-fragmentation equations with non-increasing mass for a wide set of coagulation and breakage functions \cite{giri2012weak, giri2011uniqueness}. In the study, there are two important ways to approach the C-F equations which are deterministic, other is stochastic.\\

In the past few years, several researchers concerned about the existence and uniqueness theory of PBEs with singular coagulation and fragmentation kernels, see \cite{saha2015singular, cueto2013singular, camejo2015regular}. In \cite{cueto2013singular, camejo2015regular}, researchers have done work by considering the bound over the coagulation kernel as $K(x,y)=\beta \frac{(1+x)^{\lambda}(1+y)^{\lambda}}{(xy)^{\sigma}}$, where $\beta{>0}$ is a constant, $\sigma \in [0,\frac{1}{2}[$, $\lambda-\sigma \in [0,1[$ and the singularity occurs on the both coordinate axes, i.e. $x=0, y=0.$ The bound over the selection rate is considered to be $S(x)\leq x^{\theta}$, where $\theta \in [0,1[.$ Further, in  \cite{saha2015singular}, the author extended the above work for the range of $\lambda-\sigma \in [0,1].$ Here, coagulation kernel contains a wide set of functions as compared to the kernels used in \cite{cueto2013singular, camejo2015regular}. Recently, in \cite{barik2020existence}, the author has discussed the existence of mass-conserving weak solutions to the continuous coagulation and multiple fragmentation equation by extending the previous results in \cite{cueto2013singular, camejo2015regular}, by considering $\lambda-\sigma \in [0, 1]$ for the linear coagulation rate to the large size particles and singularities for smaller ones.\\

From the numerical point of view, Bourgade and Filbet \cite{bourgade2008convergence} have treated binary coagulation-fragmentation equations using finite volume approximation. They described the convergence of numerically discretized solution towards a weak solution of the continuous model with locally bounded kernels in weighted $L^1$ space. Kumar and Kumar\cite{kumarfinite} extended the results for multiple fragmentation by following similar convergence analysis, and having locally bounded breakage kernels. Since, all the existing literature discussed the convergence analysis for locally bounded kernels, and having the availability of mathematical result on singular kernels in \cite{barik2020existence}, it would be interesting to study numerically, this convergence analysis for finite volume scheme having singularity near the origin for the fragmentation kernels.\\

Therefore, our work is inspired by taking the coupled problems mentioned above. The aim is to establish weak convergence analysis of the finite volume scheme (FVS) for the coagulation and multiple fragmentation equations having singular multiple breakage kernel and locally bounded coagulation kernels. Thanks to the Dunford-Pettis theorem for the use of weak $L^1$ compactness argument and a refined version of De La Vall$\acute{e}$e Poussin theorem.\\

We demonstrate the convergence analysis in a weighted $L^1$ space $X^+$ given by
\begin{align*}
X^{+} = \{c\in L^{1}(\mathbb{R}_{>0})\cap L^1(\mathbb{R}_{>0}, x\,dx): c\geq 0, \|c\|< \infty\}
\end{align*}
where $\|c\|= \int_{0}^{\infty}(1+x)|c(x)|\,dx$ and taking the non-negative initial condition $c^{in} \in X^{+}$. The notation $L^1(\mathbb{R}_{>0}, xdx)$ stands for the space of the Lebesgue measurable real valued functions on $\mathbb{R}_{>0}$ which are integrable with respect to the measure $x\, dx.$ Furthermore, we demonstrate theoretical and numerical error estimates for a uniform mesh using kernels in $W_{loc}^{1,\infty}$ space. It is shown that the FVS yields first-order error estimates which is verified numerically for several test problems.\\

The article is organized as follows. The non-conservative formulation of the combined coagulation and multiple fragmentation equations is discussed in the next Section \ref{divergnecechap2} together with the numerical approximations. Further in Section \ref{convergencechap2}, the main result of convergence analysis is explained for the approximated solutions using the weak compactness argument. We discuss first-order  error estimation for uniform mesh in Section \ref{errorde}. Simultaneously, the validation of error approximation is verified by some numerical simulations in Section \ref{errornu}. Finally, conclusions are made in Section 6.		

\section{Non-conservative  Formulation}\label{divergnecechap2}
	Thanks to the Leibniz integral rule, PBEs (1) for continuous coagulation-multiple fragmentation model is written as a divergence form in terms of the mass density $xc(t,x)$ as
	\begin{align}\label{chap2:equation_1st}
		\frac{x\partial c(t,x)}{\partial t}=-{\frac{\partial \mathcal C(c)(t,x)}{\partial x}}+{\frac{\partial \mathcal F(c)(t, x)}{\partial x}}, \quad (t,x)\in \mathbb{R}_{>0}^2:=]0,\infty[^2
	\end{align}
	where the continuous fluxes are being taken as
	\begin{align}\label{Coagfluxes}	
		{ \mathcal C(c)(t,x)}:= \int_0^x \int_{x-u}^\infty uK(u,v)c(t,u)c(t,v) dvdu,
	\end{align}
	for aggregation and for the breakage, the following form is obtained
	\begin{align}\label{Fragfluxes}
		{ \mathcal F(c)(t, x)}:= \int_0^x \int_x^\infty uB(u,v)S(v)c(t,v)dvdu.
	\end{align}
Throughout the article, it is assumed that the aggregation kernel $K$, the selection rate $S$ and the breakage function $B$ satisfy the following assumptions
	\begin{align}\label{aggregation funcn}
K\in L_{\text{loc}}^\infty{(\mathbb{R}_{>0} \times \mathbb{R}_{>0}) },
	\end{align}
	\begin{equation} \label{selectionrate}
S \in L_{\text{loc}}^\infty(\mathbb{R}_{>0})\  \text{and}\  S(x)=x^{1+\alpha},
	\end{equation}
	and
	\begin{equation}\label{breakage funcn}
		B(x,y)=\frac{\alpha+2}{y}\bigg(\frac{x}{y}\bigg)^{\alpha}, \ \text{for} \ 0<x <y,
	\end{equation}
	where $\alpha \in (-1, 0]$. The above class of selection rate and breakage function is a particular class of the following kinetic parameters $S(x)=x^{\gamma}$ and $B(x,y)=\frac{\alpha+2}{y}(\frac{x}{y})^{\alpha}$ for $0<x<y$ where $\gamma\in \mathbb{R}$ and $\alpha>-2$. By assuming $\gamma>0$ and $\alpha >-1$, we obtain $\eta=(\alpha+2)/(\alpha+1)$. It should be mentioned that the case of $-1<\alpha<-2$ gives a non realistic case as $\eta$ becomes infinity.\\
	
	It is clearly seen that the above conditions on selection rate and breakage function extends the case considered in \cite{kumarfinite} where the authors have put the locally bounded condition on the product of breakage and selection functions. Here, in this work, the singularity near the origin can be handled for the breakage parameter along with the coagulation problem. \\
	
	Now, in the next subsection, a numerical method to solve the equation (\ref{chap2:equation_1st}) is described. For this a finite volume approximation \cite{eymard2000finite} is taken for the volume variable $x$ while an explicit Euler method is used to discretize the time variable $t$.
	
	\subsection{Numerical Approximation}\label{numapproxchap2}
	In order to consider the more realistic case, in this section, a non-conservative truncation for the coagulation part is taken as
	\begin{align}\label{aggflux_nc}	
		{ {\mathcal C}_{nc}^{R}(c)(t,x)}:= \int_0^x \int_{x-u}^{R} uK(u,v)c(t,u)c(t,v) dvdu
	\end{align}
	by replacing $\infty$ for a positive real constant $R$ in the equation (\ref{Coagfluxes}). While for the breakage, using $R$ and equation (\ref{Fragfluxes}), a conservative approximation is considered as
	\begin{align}\label{brkflux_c}
		{ {\mathcal F}_{c}^{R}(c)(t, x)}:= \int_0^x \int_x^{R} uB(u,v)S(v)c(t,v)dvdu.
	\end{align}
	Thus, the coupled non-conservative form of the truncation is governed by
	\begin{align}\label{main}
		\left\{
		\begin{array}{lll}
			x\frac{\partial
				c}{\partial t}=-\frac{\partial
				\mathcal{C}_{nc}^R(c)}{\partial x}+\frac{\partial
				\mathcal{F}_{c}^R(c)}{\partial x}, & \hbox{$(t,x)\in \mathbb{R}_{>0}\times ]0,R]$;} \\
			&\hbox{}\\
			c(0,x)=c^{in}(x),& \hbox{$x\in ]0,R]$.}
		\end{array}
		\right.
	\end{align}
	Such truncation is chosen so that it enables for the simulation of gelation phenomena. Although, it depends on mainly the higher rate of kernels $K$ and $B$, it should be mentioned here that flux (\ref{aggflux_nc}) leads to the decrease in total mass in the system while expression (\ref{brkflux_c}) yields the total mass conservation. One can easily verify these by having
	$$\frac{d}{d t}\int_0^R x c(t,x)dx= - \mathcal{C}_{nc}^{R}(c)(t,R) \leq 0$$ in case of pure coagulation and
	$$\frac{d}{d t}\int_0^R x c(t,x)dx = 0$$ for the pure breakage case.
	
	Now, to apply the numerical scheme and to discretize the volume variable of the equation (\ref{main}), let $h\in]0,1[ $, $\mathrm{I}^h$ a positive integer such that $(x_{i-1/2})_{i\in \{0,\ldots,\mathrm{I}^h\}}$ is a mesh of $]0,R]$ having
	$$x_{-1/2}=0, \ \ x_{\mathrm{I}^h+1/2}= R, \ \ x_i=(x_{i-1/2}+x_{i+1/2})/2, \ \ \Delta x_i=x_{i+1/2}-x_{i-1/2}\leq
	h$$ and $\Lambda_i^h=]x_{i-1/2},x_{i+1/2}]$ for $i\geq 0$.
	For given integers $i$ and $j$ such that $x_{i+1/2}-x_{j}\geq 0$, define integer $\gamma_{i, j} \in \{0,\ldots,\mathrm{I}^h\}$ such that
	\begin{align*}
		x_{i+1/2}-x_{j}\in \Lambda_{\gamma_{i,j}}^h.
	\end{align*}
	For non-uniform mesh, introduce $\delta_h=\min\Delta x_i$ and consider a positive constant $L$ as
		
		\begin{align}\label{meshcondition}
			\frac{h}{\delta_h} \le L,
		\end{align}
		while for the uniform mesh, i.e.\ $\Delta x_i=h\ \forall\, i$, one obtain that $x_{i-1/2}=ih$ and $\gamma_{i,j}=i-j$.
	
Further, for discretizing the time variable $t$, assuming $\Delta t$ being the time step and $[0,T]$ is the time domain, we have $N\Delta t=T$ for a large positive integer $N$. Define the time interval $$\tau_n :=[t_n,t_{n+1}[$$ having $t_n=n\Delta t,\,n\geq 0$.\enter
	
By having the above discretizations for volume variable $x$ and time $t$, let us begin with the finite volume scheme studying for the given equation. Consider the approximation of $c(t,x)$ as $c_i^n$ for $t\in \tau_n$ and $x\in \Lambda_i^h$  for all $i\in \{0,\ldots, \mathrm{I}^h \}$ and $n\in \{0,\ldots, N-1\}$. Further, for the kinetic parameters $K,B$, and $S$, for the time being, assuming the discretized form as $K(u,v)\approx K^h(u,v)=K_{j,i}$, $S(v)\approx S^h(v)= S_i$ and $B(u,v)\approx B^h(u,v)=  B_{j,i}$ for $v\in \Lambda_i^h$ and $u\in \Lambda_j^h$.\enter	
	
Now, integrating equation (\ref{main}) with respect to the variables $x$ and $t$ over a cell in space $\Lambda_i^h$ and time $\tau_n$, respectively leads to
	  \begin{align*}
	  	\int_{t_n}^{t_{n+1}}\int_{x_{i-1/2}}^{x_{i+1/2}} \frac{\partial
	  		(xc(t,x))}{\partial t}dx\,dt=-\int_{t_n}^{t_{n+1}} \int_{x_{i-1/2}}^{x_{i+1/2}} \frac{\partial
	  		\mathcal{C}_{nc}^R(c)(t, x)}{\partial x} dx\,dt+ \int_{t_n}^{t_{n+1}} \int_{x_{i-1/2}}^{x_{i+1/2}} \frac{\partial
	  		\mathcal{F}_{c}^R(c)(t, x)}{\partial x} dx\,dt.
	  \end{align*}
The above yields the following discretized form of the equation for the coagulation and multiple fragmentation processes
	  	\begin{align}\label{gun}
	  		\Delta x_i x_i (c_i^{n+1}- c_i^n) =- \Delta t\left(\mathcal{C}_{i+1/2}^n-\mathcal{C}_{i-1/2}^n\right)+\Delta t\left(\mathcal{F}_{i+1/2}^n-\mathcal{F}_{i-1/2}^n\right)
	  	\end{align}
where $\mathcal{C}_{i+1/2}^n$ and $\mathcal{F}_{i+1/2}^n$ are the numerical fluxes which are approximations of the continuous flux functions, respectively for $\mathcal{C}_{nc}^R(c)(x)$ and $\mathcal{F}_{c}^R(c)(x)$. Therefore, these are computed as
\begin{align}\label{numfluxaggre}
	  		\mathcal{C}_{nc}^R(c)(x_{i+1/2})=& \int_0^{x_{i+1/2}} \int_{x_{i+1/2}-u}^R u K(u,v)c(u) c(v) dv\, du \nonumber\\
	  		=& \sum_{j=0}^{i} \int_{\Lambda_j^h} uc(u) \sum_{K=\gamma_{i,j}}^{\mathrm{I}^h} \int_{\Lambda_k^h}  K(u,v)c(v) dv\, du \nonumber\\
	  		\approx& \sum_{j=0}^{i}\sum_{k=\gamma_{i,j}}^{\mathrm{I}^h}  x_j  K_{j,k} c_j^n c_k^n \Delta x_j \Delta x_k := \mathcal{C}_{i+1/2}^n.
\end{align}
Similarly, for the breakage, 	  		
	  		\begin{align}\label{numfluxbrk}
	  			\mathcal{F}_c^R(c)(x_{i+1/2})=& \int_0^{x_{i+1/2}} \int_{x_{i+1/2}}^R u B(u,v)S(v) c(v) dv\, du \nonumber\\
	  			=& \sum_{j=0}^{i} \int_{\Lambda_j^h}  \sum_{k=i+1}^{\mathrm{I}^h} \int_{\Lambda_k^h} u S(v)c(v) B(u,v) dv\, du \nonumber\\
	  			\approx& \sum_{j=0}^{i} \sum_{k=i+1}^{\mathrm{I}^h} x_j S_k B_{j,k} c_k^n \Delta x_k \Delta x_j := \mathcal{F}_{i+1/2}^n.	
	  	\end{align}
Also, the initial condition is approximated as
	  	$$c_i^{in}=\frac{1}{\triangle x_i}\int_{\Lambda_i^h}c^{in}(x)dx, \quad i\in \{0,\ldots, \mathrm{I}^h\}.$$
Let us denote the characteristic function $\chi_D(x)$ of a set $D$ as $\chi_D(x)=1$ if $x\in D$ or $0$ everywhere else. Then a function $c^h$ on $[0,T]\times ]0,R]$ is defined as
	  	 \begin{align}\label{chap2:function_ch}
	  	 	c^h(t,x)=\sum_{n=0}^{N-1}\sum_{i=0}^{\mathrm{I}^h}c_i^n\,
	  	 	\chi_{\Lambda_i^h}(x)\,\chi_{\tau_n}(t),
	  	 \end{align}
	  	 which means that the function $c^h$ relies on the volume and time steps. Also noting that $$c^h(0,\cdot)=\sum_{i=0}^{\mathrm{I}^h}c_i^{in} \chi_{\Lambda_i^h}(\cdot),$$
	  	 converges strongly to $c^{in}$ in
	  	 $L^1]0,R]$ as $h\rightarrow 0$. Further, the following forms of aggregation, fragmentation and selection functions in discrete setting are taken
	  	 \begin{align}\label{chap2:function_aggregatediscrete}
	  	 	K^h(u,v)= \sum_{i=0}^{\mathrm{I}^h} \sum_{j=0}^{\mathrm{I}^h} K_{i,j} \chi_{\Lambda_i^h}(u) \chi_{\Lambda_j^h}(v)\quad \text{where}\quad K_{i,j}= \frac{1}{\Delta x_i \Delta x_j} \int_{\Lambda_j^h} \int_{\Lambda_i^h} K(u,v)du dv,
	  	 \end{align}
	  	 \begin{align}\label{chap2:function_brkdiscrete}
	  	 	B^h(u,v)= \sum_{i=0}^{\mathrm{I}^h} \sum_{j=0}^{\mathrm{I}^h} B_{i,j} \chi_{\Lambda_i^h}(u) \chi_{\Lambda_j^h}(v)\quad \text{where}\quad B_{i,j}= \frac{1}{\Delta x_i \Delta x_j} \int_{\Lambda_j^h} \int_{\Lambda_i^h} B(u,v)du dv,
	  	 \end{align}
	  	 and
	  	 \begin{align}\label{chap2:function_selectiondiscrete}
	  	 	S^h(v)= \sum_{i=0}^{\mathrm{I}^h} S_i \chi_{\Lambda_i^h}(v)\quad \text{where}\quad S_{i}= \frac{1}{\Delta x_i} \int_{\Lambda_i^h} S(v) dv.
	  	 \end{align}
Such discretization ensures that {$\|K^h-K\|_{L^1(]0,R]\times ]0,R])}\rightarrow 0$}, {$\|B^h-B\|_{L^1(]0,R]\times ]0,R])}\rightarrow 0$} and $\|S^h-S\|_{L^1(]0,R])}\rightarrow 0$ as $h\rightarrow 0$.
	  	
\section{Convergence of Solutions}\label{convergencechap2}
Below, the main findings of this work, i.e. the convergence of truncated solutions towards a weak solution of the continuous problem (\ref{main}) is discussed.
	  	
	  \begin{thm}\label{maintheorem}	 Assume that $c^{in}\in X^+$. Let the coagulation kernel $K$, the fragmentation  function $B$ and the selection rate $S$ satisfy, ({\ref{aggregation funcn}}), ({\ref{breakage funcn}}), and ({\ref{selectionrate}}), respectively. Also assuming that under the time step $\Delta t$ and for a constant $\theta> 0$, the following stability condition\\
	  		\begin{align}\label{22}
	  			C(R, T)\Delta t\le \theta< 1,
	  		\end{align}
	  		holds for
	  		\begin{align}{\label{23}}
	  			C(R, T):=\max(M,\alpha + 2)\, \max(\|K\|_{\infty},1)\|c^{in}\|_{L^1}\,e^{\eta \|S\|_{L^{\infty}} T} + R^{1+\alpha}\eta.
	  		\end{align}
	  	
	  	Then up to the extraction of a subsequence, $$c^h\rightarrow c\ \
	  	\text{in}\ \ L^\infty(0,T;L^1\,]0,R]),$$
	  	for c being the weak solution to (\ref{main}) on $[0,T]$ with initial
	  	condition $c^{in}$. This implies that, the function $c\geq 0$ satisfies
	  	\begin{equation}
	  	\begin{aligned}\label{240800}
	  {}	&	\int_0^T\int_0^R xc(t,x)\frac{\partial\varphi}{\partial
	  			t}(t,x)dx\,dt +\int_0^R xc^{in}(x)\varphi(0,x)dx \\
	  		& + \int_0^T \int_0^R [\mathcal{C}_{nc}^R(t,x)-\mathcal{F}_{c}^R(t,x)] \frac{\partial
	  			\varphi}{\partial x}(t,x)dx\,dt=\int_0^T \mathcal{C}_{nc}^R(t,R)\varphi(t,R) dt
	  	\end{aligned}
	  	\end{equation}
	  	for $\varphi$ being the continuously differentiable functions having compact support in $[0,T[\times [0,R].$
	  \end{thm}
	  By following the above theorem, it is certain that the main task here is to prove the weak convergence of the family of functions $(c^h)_{h \in (0, 1)}$ to a
	  function $ c $ in $L^1]0,R]$ as $h,\Delta t \rightarrow 0$. The idea is to use a necessary and sufficient condition for compactness with respect to the weak convergence in $L^1$ given by the following Dunford-Pettis theorem.	  	
	  	\begin{thm}
	  		Let us take $|\Omega|<\infty$ and $c^h:\Omega\mapsto \mathbb{R}$ be a sequence in $L^1(\Omega).$
	  		Assume that the sequence $\{c^h\}$ follows:
	  		\begin{itemize}
	  			\item $\{c^h\}$ is equibounded in $L^1(\Omega)$, i.e.\
	  			\begin{align}\label{equiboundedness}
	  				\sup \|c^h\|_{L^1(\Omega)}< \infty
	  			\end{align}
	  			\item $\{c^h\}$ is equiintegrable, iff
	  			\begin{align}\label{equiintegrable}
	  				\int_\Omega \Phi(|c^h|)dx< \infty
	  			\end{align}
	  			for  $\Phi$ being some increasing function taken as $\Phi:[0,\infty[\mapsto [0,\infty[$
	  			such that
	  			\begin{align*}
	  				\lim_{r\rightarrow \infty}\frac{\Phi(r)}{r}\rightarrow
	  				\infty.
	  			\end{align*}
	  		\end{itemize}
	  		Then $c^h$ belongs to a weakly compact set in $L^1(\Omega)$, implying that there exists a subsequence of $c^h$ that converges weakly in $L^1(\Omega)$.
	  	\end{thm}
	  	
	  	Hence, to prove the theorem \ref{maintheorem}, it is enough to establish the equiboundedness and the equiintegrability of the family $c^h$ in $L^1$ as given in (\ref{equiboundedness}) and (\ref{equiintegrable}), respectively. Let us begin with the proof of non-negativity and the equiboundedness of the function $c^h$ in the following proposition. In order to proceed this, let us denote $X^h(x)$ as a mid-point approximation of a point $x$, i.e.\ $X^h(x)=x_i$
	  	for $x\in \Lambda_i^h.$
\begin{propos}
{ Let us consider that the stability condition (\ref{22}) holds for the time step $\Delta t$. Also, assuming that the growth conditions on kernels satisfy (\ref{aggregation funcn})-(\ref{breakage funcn}). Then $c^h$ is a non-negative function such that
\begin{align}
\int_0^R X^h(x)c^h(t,x)dx \leq \int_0^R X^h(x)c^h(s,x)dx \leq \int_0^R X^h(x)c^h(0,x)dx=:\mu_1^{in},
\end{align}
where $ 0\le s\le t\le T.$ Then the following estimates obtained
\begin{align}\label{36}
\int_0^R c^h(t,x)dx \le  \|c^{in}\|_{L^1}\,e^{\eta \|S\|_{L^{\infty}} t}.
\end{align}
}
\end{propos}
	  		
\begin{proof}
The non-negativity and the equiboundedness of the function $c^h$ are shown here by using induction. It is known  that at $t=0$, $c^h(0)\geq 0$ and belongs to $L^1]0,R].$ Assuming further that the function $c^h(t^n)\ge 0$ and
\begin{align}\label{mon1}
\int_0^R c^h(t^n,x)dx \le  \|c^{in}\|_{L^1}\,e^{\eta \|S\|_{L^{\infty}} t^n}.
\end{align}
	  				
Then, our first aim is to prove that $c^h(t^{n+1})\geq 0.$ Firstly, consider the cell at the boundary having index $i=0$.
By (\ref{numfluxaggre}) and (\ref{numfluxbrk}), we have $\mathcal{C}_{i\pm1/2}^n\geq 0$, $\mathcal{F}_{i\pm1/2}^n\geq 0$. Therefore, in this case, from the equation (\ref{gun}) and by using the fluxes at boundaries, one obtains
	  		
\begin{align*}
x_0c_0^{n+1}= & x_0c_0^n-\frac{\Delta t}{\Delta x_0}\mathcal{C}_{1/2}^n+\frac{\Delta t}{\Delta x_0}\mathcal{F}_{1/2}^n\\
\ge &  x_0c_0^n-\frac{\Delta t}{\Delta x_0}\mathcal{C}_{1/2}^n\\
\ge & \bigg(1-\Delta t\sum_{k=0}^{\mathrm{I}^h}\Delta x_kK_{0,k}c_k^n \bigg)x_0c_0^n.
	  		\end{align*}	
 Now, having the stability condition on the time step $\Delta t$ and from equation (\ref{mon1}), the non-negativity of $c_0^{n+1}$ follows. For $i \geq 1$, one has	

 	\begin{align*}
 		 x_i c_i^{n+1}= x_i c_i^n -\frac{ \Delta t}{\Delta x_i}\left(\mathcal{C}_{i+1/2}^n-\mathcal{C}_{i-1/2}^n\right)+\frac{\Delta t}{\Delta x_i}\left(\mathcal{F}_{i+1/2}^n-\mathcal{F}_{i-1/2}^n\right).
 	\end{align*}
 	
 	Following the equations (\ref{numfluxaggre}), (\ref{numfluxbrk}) and the non-negativity of $c^h(t^n)$, it implies that
 	\begin{align}
 		-\frac{\mathcal{C}_{i+1/2}^n-\mathcal{C}_{i-1/2}^n}{\Delta
 			x_i}=& \frac{1}{\Delta x_i}\bigg[-\sum_{j=0}^{i}\sum_{k=\gamma_{i,j}}^{\mathrm{I}^h}  x_j  K_{j,k} c_j^n c_k^n \Delta x_j \Delta x_k+\sum_{j=0}^{i-1}\sum_{k=\gamma_{i-1,j}}^{\mathrm{I}^h}  x_j  K_{j,k} c_j^n c_k^n \Delta x_j \Delta x_k\bigg] \nonumber\\ =& \frac{1}{\Delta x_i}\bigg[-\sum_{k=\gamma_{i,i}}^{\mathrm{I}^h} x_i  K_{i,k} c_i^n c_k^n \Delta x_i \Delta x_k+\sum_{j=0}^{i-1}\sum_{k=\gamma_{i-1,j}}^{\gamma_{i,j}-1}  x_j  K_{j,k} c_j^n c_k^n \Delta x_j \Delta x_k\bigg] \nonumber\\ \geq& -\sum_{k=0}^{\mathrm{I}^h}  \Delta x_k K_{i,k}x_i c_k^n  c_i^n,
 	\end{align}
 	and
 	\begin{align}\label{mon}
 		\frac{\mathcal{F}_{i+1/2}^n-\mathcal{F}_{i-1/2}^n}{\Delta
 			x_i}=& \frac{1}{\Delta x_i}\bigg[\sum_{j=0}^{i} \sum_{k=i+1}^{\mathrm{I}^h} x_j S_k B_{j,k} c_k^n \Delta x_j \Delta x_k- \sum_{j=0}^{i-1} \sum_{k=i}^{\mathrm{I}^h} x_j S_k B_{j,k} c_k^n \Delta x_j \Delta x_k\bigg] \nonumber\\ =& \frac{1}{\Delta x_i}\bigg[-\sum_{j=0}^{i-1} x_j S_i B_{j,i} c_i^n \Delta x_i \Delta x_j+ \sum_{k=i+1}^{\mathrm{I}^h} x_i S_k B_{i,k} c_k^n \Delta x_k \Delta x_i\bigg] \nonumber\\ \geq& -\sum_{j=0}^{i-1} S_i B_{j,i} \Delta x_j x_i c_i^n, \hspace{0.4cm} \mbox{as} ~\, x_j< x_i ~ \mbox{for}~  j< i.
 	\end{align}
By using the above bounds and then the assumptions taken in expressions (\ref{aggregation funcn}), (\ref{breakage funcn}) and (\ref{selectionrate}) lead to
 	
 	\begin{align*}
 		x_i c_i^{n+1} \geq & \bigg(1-\Delta t\bigg(\sum_{k=0}^{\mathrm{I}^h}  \Delta x_k K_{i,k} c_k^n +\sum_{k=0}^{\mathrm{I}^h} S_i B_{k,i} \Delta x_k\bigg) \bigg)x_i c_i^n\\
 		\geq& \bigg(1-\Delta t\bigg(\|K\|_{\infty}\sum_{k=0}^{\mathrm{I}^h}  \Delta x_k c_k^n +\eta R^{1+\alpha}\bigg) \bigg)x_i c_i^n, \quad \text{for}\quad \eta=\frac{\alpha+2}{\alpha+1}.
 	\end{align*}
 Hence, following the stability condition (\ref{22}) on $\Delta t$ and the $L^1$ bound (\ref{mon1}) give
  	$$c^h(t^{n+1})\geq 0.$$
 Further, by summing (\ref{gun}) over $i$ and using the fluxes at boundaries as assumed in (\ref{chapter2:brkfluxcon}), the following time monotonicity result is obtained for the total mass from the non-negativity of $c^h$ as
  	
  	\begin{align*}
  	\sum_{i=0}^{\mathrm{I}^h}\Delta {x_i} x_i c_i^{n+1} = \sum_{i=0}^{\mathrm{I}^h}\Delta x_i x_i c_i^{n}-\Delta t\, \mathcal{C}_{\mathrm{I}^h+1/2}^n \leq \sum_{i=0}^{\mathrm{I}^h}\Delta {x_i} x_i c_i^{n}.
  	\end{align*}
  		Next, it is shown that $c^h(t^{n+1})$ follows a similar estimate as
  		(\ref{mon1}). For this, multiply equation (\ref{gun}) by the term $\Delta x_i/x_i$ and using summation with respect to $i$, provide
  		
  		\begin{align}\label{gun1}
  		\sum_{i=0}^{\mathrm{I}^h} \Delta	{x_i} c_i^{n+1}= \sum_{i=0}^{\mathrm{I}^h} \Delta	{x_i} c_i^n -\Delta t \sum_{i=0}^{\mathrm{I}^h}{\frac{\left(\mathcal{C}_{i+1/2}^n-\mathcal{C}_{i-1/2}^n\right)}{x_i}}+\Delta t \sum_{i=0}^{\mathrm{I}^h}{\frac{\left(\mathcal{F}_{i+1/2}^n-\mathcal{F}_{i-1/2}^n\right)}{x_i}}.
  		\end{align}
  		The second term on the right hand side of the above equation can be simplified as
  		\begin{align}\label{aggregationcondition}
  			 -
  			 \sum_{i=0}^{\mathrm{I}^h}{\frac{\left(\mathcal{C}_{i+1/2}^n-\mathcal{C}_{i-1/2}^n\right)}{x_i}} \leq - \sum_{i=0}^{\mathrm{I}^h}{\mathcal{C}_{i+1/2}^n \bigg(\frac{1}{x_i}-\frac{1}{x_{i+1}}\bigg)} \leq 0,~ ~\mbox{due}~ \mbox{to}~ \mathcal{C}_{i+1/2}^n \geq 0 ~\forall \,i.
  			\end{align}
  			Now, for the fragmentation term, we have  			
  			\begin{align*}
  				\sum_{i=0}^{\mathrm{I}^h}\frac{\mathcal{F}_{i+1/2}^n-\mathcal{F}_{i-1/2}^n}{x_i}\leq
  				\sum_{i=0}^{\mathrm{I}^h}\sum_{k=i+1}^{\mathrm{I}^h} \Delta x_i \Delta x_k S_k B_{i,k} c_k^n.
  			\end{align*}
  			
  			By changing the order of summation and then using the conditions (\ref{selectionrate}) and (\ref{breakage funcn}) on selection rate and fragmentation function, it is easy to see that
  				
  				\begin{align}\label{breakagecondition}
  					\sum_{i=0}^{\mathrm{I}^h}\frac{\mathcal{F}_{i+1/2}^n-\mathcal{F}_{i-1/2}^n}{x_i}\le &
  					\sum_{k=0}^{\mathrm{I}^h}\Delta x_k  c_k^n \sum_{i=0}^{k-1}\Delta x_i S_k B_{i,k}\nonumber\\
  					\le & \|S\|_{L^{\infty}} \sum_{k=0}^{\mathrm{I}^h}\Delta x_k  c_k^n \int_{0}^{x_k} B(x,x_k) dx \leq \|S\|_{L^{\infty}} \frac{\alpha+2}{\alpha+1} \sum_{k=0}^{\mathrm{I}^h}\Delta x_k  c_k^n.
  				\end{align}
  	By using (\ref{aggregationcondition}) and (\ref{breakagecondition}) into (\ref{gun1}), we find that for $\eta= \frac{\alpha+2}{\alpha+1}$
  		\begin{align*}
  			\sum_{i=0}^{\mathrm{I}^h}\Delta x_i c_i^{n+1} \le    (1+\eta \|S\|_{L^{\infty}} \Delta t) \sum_{i=0}^{\mathrm{I}^h}\Delta x_i c_i^n.
  		\end{align*}
  		Finally, having (\ref{mon1}) at step $n$ and the relation $1+x < \exp(x)$ for all $x>0$ provides
  		\begin{align*}
  			\sum_{i=0}^{\mathrm{I}^h}\Delta x_i c_i^{n+1} \le  (1+\eta \|S\|_{L^{\infty}} \Delta t) \,{\|c^{in}\|}_{L^1}\,e^{\eta \|S\|_{L^{\infty}} t^n}
  			\leq {\|c^{in}\|}_{L^1} \,e^{\eta \|S\|_{L^{\infty}}  t^{n+1}},
  		\end{align*}
  		and therefore the result (\ref{36}) follows.
 	 \end{proof}
	
	   In order to prove uniform integrability of the family of solutions, let us denote a particular class of convex functions  as ${C}_{VP, \infty}$. Further, consider $\Phi \in {C}^{\infty}([0, \infty))$, a non-negative and convex function which belongs to the class ${C}_{VP, \infty}$ and enjoys the following properties:
	   \begin{description}
	   	\item[(i)] $\Phi(0)=0,\ \Phi'(0)=1$ and $\Phi'$ is concave;
	   	\item[(ii)] $\lim_{p \to \infty} \Phi'(p) =\lim_{p \to \infty} \frac{ \Phi(p)}{p}=\infty$;
	   	\item[(iii)] for some $\lambda \in (1, 2)$,
	   	\begin{align}\label{Tproperty}
	   		T_{\lambda}(\Phi):= \sup_{p \ge 0} \bigg\{   \frac{ \Phi(p)}{p^{\lambda}} \bigg\} < \infty.
	   	\end{align}
	   \end{description}
Example of such a ${C}_{VP, \infty}$ function is $\Phi(p) =2(1+p) \ln(1+p) -p$. Since, $c^{in}\in L^1\,]0,R]$, therefore, by the La Vall\'{e}e Poussin theorem, see (\cite{Laurencot:2015}, Theorem 2.8), a continuously differentiable convex function
	   $\mathrm{\Phi}\geq 0$ exists in $\mathbb{R}_{>0}$ with properties $\mathrm{\Phi}(0)=0$, $\mathrm{\Phi}'(0)=1$
	   such that $\mathrm{\Phi}'$ is concave,
	   $$\frac{\mathrm{\Phi}(p)}{p} \rightarrow \infty,\ \ \text{as}\ \
	   p \rightarrow \infty$$ and
	   \begin{align}\label{convex}
	   	\mathcal{I}:=\int_0^R \mathrm{\Phi}(c^{in})(x)dx< \infty.
	   \end{align}
	
	   \begin{lem} [\cite{laurenccot2002continuous}, Lemma B.1.]
	   	Let $\mathrm{\Phi}\in {C}_{VP, \infty}$. Then the estimates
	   	$$x\mathrm{\Phi}'(y)\leq \mathrm{\Phi}(x)+\mathrm{\Phi}(y)$$ holds $\forall\ (x,y)\in \mathbb{R}_{>0}\times \mathbb{R}_{>0}$.
	   	\end{lem}
	   	
	  	Now, in the following proposition, the equi-integrability is discussed.
	  	\begin{propos}\label{equiintegrability}
	  		Let $c^{in}\geq 0\in L^1 ]0,R]$ and the family $(c^h)_{(h,\Delta t)}$ is defined for all $h$ and $\Delta t$ given by (\ref{chap2:function_ch}) where $\Delta t$ satisfies the relation (\ref{22}). Then $(c^h)$ is weakly relatively sequentially compact in $L^1(]0,T[\times ]0,R])$.
	  	\end{propos}
	  	\begin{proof}
	  		Our focus here is to obtain a similar result as (\ref{convex}) for the family of function $c^h$. The integral of $\mathrm{\Phi}(c^h)$ by using the
	  		sequence $c_i^n$ can be written as
	  		\begin{align*}
	  			\int_0^T \int_0^R \mathrm{\Phi}(c^h(t,x))dx\,dt=&\sum_{n=0}^{N-1}\sum_{i=0}^{\mathrm{I}^h}\int_{\tau_n}\int_{\Lambda_i^h}\mathrm{\Phi}
	  			\bigg(\sum_{k=0}^{N-1}\sum_{j=0}^{\mathrm{I}^h}c_j^k\chi_{\Lambda_j^h}(x)\chi_{\tau_k}(t)\bigg)dx\,dt\\
	  			=&\sum_{n=0}^{N-1}\sum_{i=0}^{\mathrm{I}^h}\Delta t\Delta x_i\mathrm{\Phi}(c_i^n).
	  		\end{align*}
	  		The convexity of the function $\mathrm{\Phi}$ leads to the estimate
	  		\begin{align*}
	  			\left(c_i^{n+1}-c_i^{n}\right) \mathrm{\Phi}^{'}(c_i^{n+1})\geq \mathrm{\Phi}(c_i^{n+1})-\mathrm{\Phi}(c_i^{n}).
	  		\end{align*}
	  		Now, multiply the above equation by $\Delta x_i$ and summing with respect to $i$ on both sides yield
	  		\begin{align*}
	  			\sum_{i=0}^{\mathrm{I}^h} \Delta x_i \left[\mathrm{\Phi}(c_i^{n+1})-\mathrm{\Phi}(c_i^{n})\right]\leq &\sum_{i=0}^{\mathrm{I}^h} \Delta x_i \left[(c_i^{n+1}-c_i^{n})\mathrm{\Phi}^{'}(c_i^{n+1})\right].
	  		\end{align*}
By using discrete coagulation and fragmentation terms, it can further be rewritten as	  	
\begin{align}\label{integrability}
\sum_{i=0}^{\mathrm{I}^h} {\Delta x_i \left[\mathrm{\Phi}(c_i^{n+1})-\mathrm{\Phi}(c_i^{n})\right]} \le & {\|K\|}_{L^\infty} \Delta t \sum_{i=0}^{\mathrm{I}^h}\sum_{j=0}^{i-1}\Delta{ x_j}c_j^n \sum_{k=\gamma_{i-1,j}}^{\gamma_{i,j}-1}\Delta {x_k}c_k^n\mathrm{\Phi}^{'}(c_i^{n+1}) \nonumber \\
& +   \Delta t \sum_{i=0}^{\mathrm{I}^h}  \sum_{k=i+1}^{\mathrm{I}^h}  S_k B_{i,k} c_k^n \Delta x_k \Delta x_i \mathrm{\Phi}^{'}(c_i^{n+1}).
\end{align}
For the coagulation term, the first part of the right hand side of the above equation, by changing the order of summation and the convexity result given as in Lemma 3.4, we have

\begin{align*}
	 {\|K\|}_{L^\infty} \Delta t & \sum_{i=0}^{\mathrm{I}^h}\sum_{j=0}^{i-1}\Delta{ x_j}c_j^n \sum_{k=\gamma_{i-1,j}}^{\gamma_{i,j}-1}\Delta {x_k}c_k^n\mathrm{\Phi}^{'}(c_i^{n+1})\\
	 \le & {\|K\|}_{L^\infty} \Delta t \sum_{j=0}^{\mathrm{I}^h}\Delta{ x_j}c_j^n \sum_{i=j+1}^{\mathrm{I}^h} \sum_{k=\gamma_{i-1,j}}^{\gamma_{i,j}-1}\Delta {x_k}[\mathrm{\Phi}(c_k^{n})+\mathrm{\Phi}(c_i^{n+1})].
	\end{align*}
  It is easy to see that,
\begin{align}\label{39}
	\sum_{i=j+1}^{\mathrm{I}^h} \sum_{k=\gamma_{i-1,j}}^{\gamma_{i,j}-1}\Delta {x_k}\mathrm{\Phi}(c_k^{n})= \sum_{k=\gamma_{j,j}}^{\gamma_{\mathrm{I}^h,j}-1}\Delta {x_k}\mathrm{\Phi}(c_k^{n}) \leq \sum_{k=0}^{\mathrm{I}^h}\Delta {x_k}\mathrm{\Phi}(c_k^{n}).
	\end{align}	
	Further, to simplify the term \[ \sum_{i=j+1}^{\mathrm{I}^h} \sum_{k=\gamma_{i-1,j}}^{\gamma_{i,j}-1}\Delta {x_k}\mathrm{\Phi}(c_i^{n+1}),\] we proceed as follows. Notice that
	\begin{align*}
\sum_{k=\gamma_{i-1,j}}^{\gamma_{i,j}-1}\Delta {x_k}= x_{\gamma_{i,j}-1/2}-x_{\gamma_{i-1,j}-1/2}.
	\end{align*}
	where $x_{\gamma_{i,j}-1/2}$ and $x_{\gamma_{i-1,j}-1/2}$ are being the left point approximations of $x_{i+1/2}-x_{j}$ and $x_{i-1/2}-x_{j}$, respectively. Hence, the  following inequality holds
	\begin{align*}
		x_{\gamma_{i,j}-1/2}-x_{\gamma_{i-1,j}-1/2} \leq (x_{i-1/2}-x_{j})-x_{\gamma_{i-1,j}-1/2}+\Delta {x_i}.
		\end{align*}
		Using the assumption on mesh (\ref{meshcondition}), one has
		\begin{align*}
			(x_{i-1/2}-x_{j})-x_{\gamma_{i-1,j}-1/2}\leq h \leq L\Delta x_i,
		\end{align*}
		while taking (14) leads to
		\begin{align*}
			(x_{i-1/2}-x_{j})-x_{\gamma_{i-1,j}-1/2}\leq \Delta x_i.
		\end{align*}
		Both the cases provide final result as
		\begin{align*}
			x_{\gamma_{i,j}-1/2}-x_{\gamma_{i-1,j}-1/2} \leq Q \Delta x_i
		\end{align*}
		for $Q=1+L$ or $Q=2$. Therefore,
		\begin{align}
			\sum_{i=j+1}^{\mathrm{I}^h} \sum_{k=\gamma_{i-1,j}}^{\gamma_{i,j}-1}\Delta {x_k}\mathrm{\Phi}(c_i^{n+1}) \leq Q \sum_{i=0}^{\mathrm{I}^h}\Delta{x_i}\mathrm{\Phi}(c_i^{n+1}).
		\end{align}
Again, using the convexity results on fragmentation term in the equation (\ref{integrability}), growth conditions (\ref{selectionrate})-(\ref{breakage funcn}) on selection and breakage functions and changing the order of summation follows,
		\begin{align}\label{convexfrag1}
			 \Delta t \sum_{i=0}^{\mathrm{I}^h}  \sum_{k=i+1}^{\mathrm{I}^h}  S_k B_{i,k} c_k^n \Delta x_k \Delta x_i \mathrm{\Phi}^{'}(c_i^{n+1})  = & (\alpha +2)\Delta t\sum_{i=0}^{\mathrm{I}^h}  \sum_{k=i+1}^{\mathrm{I}^h}  x_{i}^{\alpha} c_k^n \Delta x_k \Delta x_i \mathrm{\Phi}^{'}(c_i^{n+1})\nonumber \\
			  \leq & (\alpha +2)\Delta t\sum_{k=0}^{\mathrm{I}^h} c_k^n \Delta x_k \sum_{i=0}^{k-1}     \Delta x_i [\mathrm{\Phi}(c_i^{n+1})+\mathrm{\Phi}(x_{i}^{\alpha})] \nonumber\\
			  \leq & (\alpha +2)\Delta t\sum_{k=0}^{\mathrm{I}^h} c_k^n \Delta x_k \sum_{i=0}^{\mathrm{I}^h}  \mathrm{\Phi}(c_i^{n+1})\Delta x_i \nonumber \\
			 & + (\alpha +2)\Delta t\sum_{k=0}^{\mathrm{I}^h} c_k^n \Delta x_k \sum_{i=0}^{k-1}  \mathrm{\Phi}(x_{i}^{\alpha})\Delta x_i.
			\end{align}
			
		Let us estimate the two expressions on the right-hand side of the equation (\ref{convexfrag1}) separately. The first summation is evaluated as
			\begin{align}\label{Barik2}
				(\alpha + 2)  \Delta t \sum_{k=0}^{\mathrm{I}^h}    c_k^n \Delta x_k \sum_{i=0}^{\mathrm{I}^h}    \mathrm{\Phi}(c_i^{n+1})  \Delta x_i
				\le   (\alpha + 2)\|c^{in}\|_{L^1}\,e^{\eta \|S\|_{L^{\infty}} t}  \Delta t  \sum_{i=0}^{\mathrm{I}^h}    \mathrm{\Phi}(c_i^{n+1})  \Delta x_i.
			\end{align}
			 Using the Proposition 3.3, for the second term, we proceed as follows
			 \begin{align}\label{43}
			 	(\alpha + 2) \Delta t  \sum_{k=0}^{\mathrm{I}^h} \Delta x_k c_k^n \sum_{i=0}^{k-1} \mathrm{\Phi}( x_i^{\alpha}) \Delta x_i
			 	=&  (\alpha + 2) \Delta t \sum_{k=0}^{\mathrm{I}^h} \Delta x_k  c_k^n \sum_{i=0}^{k-1} \frac{ \mathrm{\Phi}( x_i^{\alpha})}  { x_i^{\lambda\alpha}} x_i^{\lambda\alpha}    \Delta x_i \nonumber\\
			 	\le & (\alpha + 2) \Delta t T_{\lambda}(\mathrm{\Phi}) \sum_{k=0}^{\mathrm{I}^h} \Delta x_k  c_k^n \sum_{i=0}^{k-1} x_i^{\lambda\alpha}    \Delta x_i~~ (~\mbox{by} ~\mbox{having}~(\ref{Tproperty})) \nonumber\\
			 	\le & (\alpha + 2) \Delta t T_{\lambda}(\mathrm{\Phi}) \sum_{k=0}^{\mathrm{I}^h} \Delta x_k  c_k^n \int_0^{x_k} x^{\lambda\alpha} dx \nonumber\\
			 	= & \frac{(\alpha + 2)}{(\lambda \alpha +1)} \Delta t T_{\lambda}(\mathrm{\Phi}) \sum_{k=0}^{\mathrm{I}^h}  c_k^n  x_k^{\lambda\alpha+1} \Delta{x_k}\nonumber\\
			 	\le & \frac{(\alpha + 2)}{(\lambda \alpha +1)} \Delta t T_{\lambda}(\mathrm{\Phi}) \bigg( \sum_{k=0}^{\mathrm{I}^h}   c_k^n \Delta{x_k} + \sum_{k=0}^{\mathrm{I}^h} x_k  c_k^n \Delta{x_k}  \bigg)\nonumber\\
			 	\le  & \frac{(\alpha + 2)}{(\lambda \alpha +1)} \Delta t T_{\lambda}(\mathrm{\Phi}) \bigg( \|c^{in}\|_{L^1}\,e^{\eta \|S\|_{L^{\infty}} t} + \mu_1^{in}  \bigg).
			 \end{align}
			 Consequently, all the results (\ref{39})-(\ref{43}) used in (\ref{integrability}) lead to
		\begin{align}\label{final}
			\sum_{i=0}^{\mathrm{I}^h} {\Delta x_i \left[\mathrm{\Phi}(c_i^{n+1})-\mathrm{\Phi}(c_i^{n})\right]} &\leq {\|K\|}_{L^\infty} \Delta t \sum_{j=0}^{\mathrm{I}^h}\Delta{ x_j}c_j^n \sum_{i=0}^{\mathrm{I}^h}\Delta {x_i}\mathrm{\Phi}(c_i^{n})\nonumber\\
			& + {\|K\|}_{L^\infty} Q \Delta t \sum_{j=0}^{\mathrm{I}^h}\Delta{ x_j}c_j^n \sum_{i=0}^{\mathrm{I}^h}\Delta {x_i}\mathrm{\Phi}(c_i^{n+1}) \nonumber\\
			& + (\alpha + 2)\|c^{in}\|_{L^1}\,e^{\eta \|S\|_{L^{\infty}} t}  \Delta t  \sum_{i=0}^{\mathrm{I}^h}   \Delta x_i \mathrm{\Phi}(c_i^{n+1})   \nonumber\\
			& + \frac{(\alpha + 2)}{(\lambda \alpha +1)} \Delta t T_{\lambda}(\mathrm{\Phi}) \bigg( \|c^{in}\|_{L^1}\,e^{\eta \|S\|_{L^{\infty}} t} + \mu_1^{in}  \bigg).		
		\end{align}
		
		It can further be simplified as
		\begin{align}
			\bigg(1-Q^{*} N^{*}\|c^{in}\|_{L^1}\,e^{\eta \|S\|_{L^{\infty}} T}  \Delta t   \bigg) \sum_{i=0}^{\mathrm{I}^h}   \Delta x_i \mathrm{\Phi}(c_i^{n+1}) \nonumber\\
			\le  \bigg(1+ \|K\|_{L^{\infty}}&\|c^{in}\|_{L^1}\,e^{\eta \|S\|_{L^{\infty}} T}\Delta t \bigg)\sum_{i=0}^{\mathrm{I}^h}   \Delta x_i \mathrm{\Phi}(c_i^{n}) \nonumber \\
			+  \frac{(\alpha + 2)}{(\lambda \alpha +1)}& \Delta t T_{\lambda}(\mathrm{\Phi}) \bigg( \|c^{in}\|_{L^1}\,e^{\eta \|S\|_{L^{\infty}} t} + \mu_1^{in}  \bigg),
		\end{align}
		where $Q^{*}$ and $N^{*}$ denote max$(Q,\alpha+2)$ and max$(\|K\|_{L^{\infty}},1)$, respectively.
		The former inequality implies that
		\begin{align}
			 \sum_{i=0}^{\mathrm{I}^h}   \Delta x_i \mathrm{\Phi}(c_i^{n+1}) \le A \sum_{i=0}^{\mathrm{I}^h}   \Delta x_i \mathrm{\Phi}(c_i^{n})+  B
			\end{align}
			where
			\begin{align*}
				A= \frac{\bigg(1+ \|K\|_{L^{\infty}}\|c^{in}\|_{L^1}\,e^{\eta \|S\|_{L^{\infty}} T} \Delta t\bigg)}{  \bigg(1-Q^{*} N^{*} \|c^{in}\|_{L^1}\,e^{\eta \|S\|_{L^{\infty}} T}  \Delta t   \bigg)},
			\end{align*}
			and
			\begin{align*}
				B=  \frac{(\alpha + 2)      \Delta t T_{\lambda}(\mathrm{\Phi}) \bigg( \|c^{in}\|_{L^1}\,e^{\eta \|S\|_{L^{\infty}} t} + \mu_1^{in}  \bigg) }{(\lambda \alpha +1)  \bigg(  1-Q^{*} N^{*} \|c^{in}\|_{L^1}\,e^{\eta \|S\|_{L^{\infty}} T}  \Delta t  \bigg)     }.
			\end{align*}
			Hence, the following is obtained
			\begin{align*}
				\sum_{i=0}^{\mathrm{I}^h}   \Delta x_i \mathrm{\Phi}(c_i^{n}) \le A^{n} \sum_{i=0}^{\mathrm{I}^h}   \Delta x_i \mathrm{\Phi}(c_i^{in})+  B \frac{A^{n-1} -1}{A-1}.
			\end{align*}
			By using Jensen's inequality and (\ref{convex}), finally we get
			\begin{align*}
				\int_0^R  \mathrm{\Phi}(c^{h}(t, x))\,dx \le & A^{n} \sum_{i=0}^{\mathrm{I}^h}   \Delta x_i \mathrm{\Phi} \bigg( \frac{1}{\Delta x_i } \int_{\Lambda_i^h}c^{in}(x) dx \bigg)+  B \frac{A^{n-1} -1}{A-1}\nonumber\\
				\le & A^{n} \sum_{i=0}^{\mathrm{I}^h}  \int_{\Lambda_i^h}  \mathrm{\Phi} (c^{in}(x) ) dx +  B \frac{A^{n-1} -1}{A-1}\nonumber\\
				= & A^{n}\mathcal{I}+  B \frac{A^{n-1} -1}{A-1} < \infty, \quad \text{for all}\ \ \ t\in [0,T].
			\end{align*}
\end{proof}
Thus, applying the Dunford-Pettis theorem, one can say that the sequence $(c^h)_{h\in (0,1)}$ is weakly compact in $L^1$, as long as the time step condition (\ref{22}) holds, the exponent
is uniformly bounded with respect to $h$ and $\Delta t$. This guarantees that a subsequence of $(c^h)_{h\in (0,1)}$ exists and $c\in L^1(]0,T[\times ]0,R])$ is such that $c^h\rightharpoonup c$ for $h\rightarrow 0$.\enter

\begin{rem}
By a diagonal procedure, subsequences of $(c^h)_h$,$(K^h)_h$ and $
(B^h)_h$ can be extracted such that
$$K^h(u,v)\rightarrow K(u,v)\quad \text{and}\quad  B^h(u,v)\rightarrow B(u,v),$$
for almost every $(u,v)\in (0,R)\times (0,R)$ as $h\rightarrow 0.$
\end{rem}
Now, we show that the discrete aggregation and  fragmentation fluxes converge weakly to the continuous fluxes  written in terms of the function $c^h$. In order to do so, some point approximations are used which are given below. Denote the midpoint approximation as
\begin{align*}
	X^h:x\in ]0,R[\rightarrow
	X^h(x)=\sum_{i=0}^{\mathrm{I}^h}x_i\chi_{\Lambda_i^h}(x).
\end{align*}
The right endpoint approximation is taken as
\begin{align*}
	\Xi^h:x\in ]0,R[\rightarrow
	\Xi^h(x)=\sum_{i=0}^{\mathrm{I}^h}x_{i+1/2}\chi_{\Lambda_i^h}(x),
\end{align*}
while for the left endpoint, we consider
\begin{align*}
	\xi^h:x\in ]0,R[\rightarrow
	\xi^h(x)=\sum_{i=0}^{\mathrm{I}^h}x_{i-1/2}\chi_{\Lambda_i^h}(x).
\end{align*}
Also, define
\begin{align*}
	\Theta^h:(x, z)\in {]0,R[}^2 \rightarrow \Theta^{h}(x, z)=\sum_{i=0}^{\mathrm{I}^h} \sum_{j=0}^{i}x_{\gamma_{i,j}}\chi_{\Lambda_i^h}(x) \chi_{\Lambda_j^h}(z).
\end{align*}
	Note that the above approximations converge pointwise, i.e.\
	$$X^h(x)\rightarrow x,\ \ \Xi^h(x)\rightarrow x\quad \text{and} \quad \xi^h(x)\rightarrow x \hspace{0.2cm} \forall\ x\in ]0,R[$$ as $h\rightarrow
	0$. Further, for all $(x, z) \in {]0,R[}^2$, one has
	\begin{center}
		$	\begin{cases}
	\Theta^h:(x, z)	\rightarrow x-z, \ \ x \ge z	,\\[.5em]
		\Theta^h:(x, z)	\rightarrow 0, \ \  x \le z.
		\end{cases}$
	\end{center}
	 Thanks to the Dunford-Pettis and Egorov theorems, the following  lemma is also needed  to show the convergence of the truncated flux towards the continuous flux.
\begin{lem}{\label{Wconverge}} [\cite{laurenccot2002continuous}, Lemma A.2]
Let $\Pi$ be an open subset of $\mathbb{R}^m$ and let there exists a constant $l>0$ and two sequences $(z^1_n)_{n\in \mathbb{N}}$ and
$(z^2_n)_{n\in \mathbb{N}}$ such that $(z^1_n)\in L^1(\Pi), z^1\in L^1(\Pi)$ and $$z^1_n\rightharpoonup z^1,\ \ \ \text{weakly in}\
\,L^1(\Pi), \text{as}\ n\rightarrow \infty,$$ $(z^2_n)\in
		L^\infty(\Pi), z^2 \in L^\infty(\Pi),$ and for all $n\in
		\mathbb{N}, |z^2_n|\leq l$ with $$z^2_n\rightarrow z^2,\ \ \text{almost
			everywhere (a.e.) in}\ \ \Pi, \ \text{as}\ \ n\rightarrow
		\infty.$$ Then
		$$\lim_{n\rightarrow \infty}\|z^1_n(z^2_n-z^2)\|_{L^1(\Pi)}=0$$
		and $$z^1_n\, z^2_n\rightharpoonup z^1\, z^2,\ \ \ \text{weakly in}\
		\,L^1(\Pi), \text{as}\ n\rightarrow \infty.$$
\end{lem}\enter
Finally, the following result explains the convergence of the numerical fluxes. For this, definitions of $c^h$, $K^h$, $B^h$ and $S^h$ given by (\ref{chap2:function_ch}),(\ref{chap2:function_aggregatediscrete}), (\ref{chap2:function_brkdiscrete}) and (\ref{chap2:function_selectiondiscrete}),
	respectively, are considered.
\begin{lem}\label{fluxconlemma}
Consider the approximations of the coagulation term as
	  	\begin{align*}
	  		\mathcal{C}^h(t,x)=\int_0^R \int_0^R
	  		\chi_{[0,\Xi^h(x)]}(u)\chi_{[\Theta^h(x,u),R]}(v) X^h(u)K^h(u,v)c^h(t,u)c^h(t,v)dvdu,
	  	\end{align*}
and for the fragmentation
	  	\begin{align*}
	  		\mathcal{F}^h(t,x)=\int_0^R \int_0^R
	  		\chi_{[0,\Xi^h(x)]}(u)\chi_{[\Xi^h(x),R]}(v) X^h(u)B^h(u,v)S^h(v)c^h(t,v)dvdu.
	  	\end{align*}
Then, there exists a subsequence of the family of $(c^h)_{h\in (0, 1)}$, such that
	  		$$\mathcal{C}^h\rightharpoonup \mathcal{C}_{nc}^R\quad \text{and}\quad \mathcal{F}^h\rightharpoonup \mathcal{F}_{c}^R$$
	 	  	in $L^1(]0,T[\times ]0,R])$ as $h\rightarrow 0.$
\end{lem}
\begin{proof} Before we begin the proof, it is important to notice that the terms $\mathcal{C}^h(t,x)$ and $\mathcal{F}^h(t,x)$ coincide with  the terms $\mathcal{C}_i^n$  and $\mathcal{F}_i^n$, respectively,  whenever $t\in \tau_n$ and $x\in
	  \Lambda_i^h$. It is easy to prove as
	
	  	\begin{align*}
	  		\mathcal{C}^h(t,x)=&\int_0^{x_{i+1/2}} \int_{\Theta^h(x,u)}^R
	  		 X^h(u)K^h(u,v)c^h(t,v)c^h(t,u)dvdu,\\
	  		  = & \sum_{j=0}^{i} \int_{\Lambda_j^h}
	  		 \sum_{k=\gamma_{i,j}}^{\mathrm{I^h}} \int_{\Lambda_k^h}\bigg[ X^h(u)\bigg(\sum_{a=0}^{\mathrm{I}^h}\sum_{b=0}^{\mathrm{I}^h}K_{a,b}\chi_{\Lambda_a^h}(u)\chi_{\Lambda_b^h}(v)\bigg)\bigg(\sum_{b=0}^{\mathrm{I}^h}c_b\chi_{\Lambda_b^h}(v)\bigg)
	  		 \\&
	  		 \times \bigg(\sum_{a=0}^{\mathrm{I}^h}c_a^n\chi_{\Lambda_a^h}(u)\bigg)\bigg]dvdu\\
	  		  = & \sum_{j=0}^{i} \sum_{k=\gamma_{i,j}}^{\mathrm{I^h}} \int_{\Lambda_j^h}\int_{\Lambda_k^h} x_j K_{j,k} c_j^n c_k^n dvdu = \mathcal{C}_{i+1/2}^n,
	  	\end{align*}
	  	while
	  	\begin{align*}
	  		\mathcal{F}^h(t,x)=& \int_{0}^{x_{i+1/2}} \int_{x_{i+1/2}}^{R}
	  		X^h(u)B^h(u,v)S^h(v)c^h(t,v)dv du\\
	  		=& \sum_{j=0}^{i} \int_{\Lambda_j^h}
	  		\sum_{k=i+1}^{\mathrm{I^h}} \int_{\Lambda_k^h}\bigg[ X^h(u)\bigg(\sum_{a=0}^{\mathrm{I}^h}\sum_{b=0}^{\mathrm{I}^h}B_{a,b}\chi_{\Lambda_a^h}(u)\chi_{\Lambda_b^h}(v)\bigg)\bigg(\sum_{b=0}^{\mathrm{I}^h}S_b\chi_{\Lambda_b^h}(v)\bigg)
	  		\\&
	  		\times \bigg(\sum_{b=0}^{\mathrm{I}^h}c_b^n\chi_{\Lambda_b^h}(v)\bigg)\bigg]dvdu\\
	  		=&\sum_{j=0}^{i} \sum_{k=i+1}^{\mathrm{I^h}} \int_{\Lambda_j^h}\int_{\Lambda_k^h} x_j B_{j,k} S_k c_k^n dvdu = \mathcal{F}_{i+1/2}^n.
	  	\end{align*}
	  	
Now, to move further, as discussed about the extraction of subsequences in Remark 3.6,
	  			we know that for $(t,x) \in ]0,T [\times ]0,R]$ and $(u,v )\in ]0,R]\times ]0,R]$ almost everywhere, the sequence $X^h(\cdot)K^h(\cdot,v)$ is bounded in $L^\infty$. Also,
	 \begin{align*}
	 	\chi_{[0,\Xi^h(x)]}(u)\chi_{[\Theta^h(x,u),R]}(v) X^h(u)K^h(u,v) \rightarrow  \chi_{[0,x]}(u)\chi_{[x-u,R]}(v) u K(u,v)
	\end{align*} 			
	as $h \rightarrow 0$. Hence, having Lemma {\ref{Wconverge}} leads to
	\begin{align}\label{coagulationconverge}
\int_0^R \chi_{[0,\Xi^h(x)]}(u) & \chi_{[\Theta^h(x,u),R]}(v) X^h(u)K^h(u,v)c^h(t,u) du \nonumber\\
 & \rightarrow \int_{0}^{R}\chi_{[0,x]}(u)\chi_{[x-u,R]}(v) u K(u,v)c(t,u) du.
	\end{align}  			
	  			
	  	The above expression says that (\ref{coagulationconverge}) holds for every $(t, x) \in ( ]0,T[\times ]0,R] )$ and almost every $v$ and $c^h$ converges weakly. Again using Lemma  {\ref{Wconverge}}, it yields
	  		\begin{align*}
	  			\mathcal{C}^h(t,x) \rightarrow \mathcal{C}_{nc}^R(t,x)
	  		\end{align*}	
	  			for every $(t,x)\in ( ]0,T[\times ]0,R] )$. Note that the weak convergence for $\mathcal{C}^h$ follows by this pointwise convergence. A similar approach shows the convergence of $\mathcal{F}^h$ as below,
	  			
	  			\begin{align*}
	  				X^h(\cdot )B^h(\cdot, v)S^h(v) = (\alpha+2) x \frac{x^{\alpha}}{v^{\alpha+1}}v^{\alpha+1} = (\alpha+2) x^{1+\alpha} \in L^{\infty}]0,R]\ \text{for almost all}\ v \in ]0,R].
	  			\end{align*}
	  			Since, the above is uniformly bounded and
	  			\begin{align*}
	  				\chi_{[0,\Xi^h(x)]}(u) \chi_{[\Xi^h(x),R]}(v)X^h(u)B^h(u,v)S^h(v) \rightarrow \chi_{[0,x]}(u) \chi_{[x,R]}(v)uB(u,v)S(v)
	  			\end{align*}
	  			pointwise almost everywhere as $h \to 0$. Hence, one has
	  			\begin{align}\label{convergefragmentation}
	  				\int_0^R & \chi_{[0,\Xi^h(x)]}(u) \chi_{[\Xi^h(x),R]}(v)X^h(u)B^h(u,v)S^h(v) du \nonumber\\
	  				& \rightarrow \int_0^R \chi_{[0,x]}(u) \chi_{[x,R]}(v)uB(u,v)S(v)du
	  			\end{align}
	  			which holds for every $(t, x) \in ( ]0,T[\times ]0,R] )$ and almost every $v$. We also know that $c^h$ weakly converges to $c$ in $L^1]0,R]$.
	  			So, applying Lemma {\ref{Wconverge}} entails
	  			\begin{align}
	  			\int_0^R & \chi_{[0,\Xi^h(x)]}(u) \chi_{[\Xi^h(x),R]}(v)X^h(u)B^h(u,v)S^h(v) c^h(t,v)dvdu \nonumber\\
	  			& \rightarrow \int_0^R \chi_{[0,x]}(u) \chi_{[x,R]}(v)uB(u,v)S(v)c(t,v)dvdu,
	  		\end{align}
	  		and therefore,
	  			\begin{align*}
	  			\mathcal{F}^h(t,x) \rightarrow  \mathcal{F}_c^R(t,x)
	  			\end{align*}
	  			for every $(t,x) \in (]0,T [\times ]0,R])$. Thanks to boundedness of $\mathcal{F}^h$, the pointwise convergence gives weak convergence.
	  		
	  		Finally, we are in position to prove the main result Theorem \ref{maintheorem} below. For the proof, a compactly supported test function $\varphi\in C^1([0,T[\times [0,R])$ is taken. The support of
	  		$\varphi$ with respect to $t$ satisfies $\text{Supp}_t \varphi\subset
	  		[0,t_{N-1}]$ for small enough time step $\Delta t$. Let us consider finite volume and left endpoint approximations for time and space variables of $\varphi$ on $\tau_n\times \Lambda_i^h$ by
	  		$$\varphi_i^n :=\frac{1}{\Delta t}\int_{t_n}^{t_{n+1}}\varphi(t,x_{i-1/2})dt.$$
	  		Now, multiply the equation (\ref{gun}) by $\varphi_i^n$, taking summation over $n\in
	  		\{0,...,N-1\}$ and $i\in \{0,...,\mathrm{I}^h\}$ lead to
	  		\begin{align*}		
	  		\sum_{n=0}^{N-1}\sum_{i=0}^{\mathrm{I}^h} \left[\Delta x_i
	  		x_i(c_i^{n+1}-c_i^n)\varphi_i^n +\Delta t \left(\mathcal{C}_{i+1/2}^n-\mathcal{C}_{i-1/2}^n\right)\varphi_i^n-\Delta t \left(\mathcal{F}_{i+1/2}^n-\mathcal{F}_{i-1/2}^n\right)\varphi_i^n\right]=0.
	  	\end{align*}
	  	Further, if we open the summation for each $i$ and $n$, then a discrete integration by parts provides
	  	\begin{align}\label{discrete24081}
	  		\sum_{n=0}^{N-1}\sum_{i=0}^{\mathrm{I}^h}\Delta
	  		x_i x_i c_i^{n+1}(\varphi_i^{n+1}-\varphi_i^n) & + \sum_{n=0}^{N-1}\sum_{i=0}^{\mathrm{I}^h-1}\Delta t [\mathcal{C}_{i+1/2}^n-\mathcal{F}_{i+1/2}^n](\varphi_{i+1}^{n+1}-\varphi_{i}^{n}) \nonumber \\
	  		 &+ \sum_{i=0}^{\mathrm{I}^h}\Delta
	  		x_i x_i c_i^{in}\varphi_i^0 - \sum_{n=0}^{N-1} \Delta
	  		t \,\mathcal{C}_{\mathrm{I}^h+1/2}^n \varphi_{\mathrm{I}^h}^n=0.
	  	\end{align}		
	  		The first and third expressions of the above equation are simplified using only the function $c^h$ while remaining terms are expressed in terms of the functions $c^h$ and $\mathcal{F}^h$. For the first part, consider
	  		\begin{align*}
	  			\sum_{n=0}^{N-1}\sum_{i=0}^{\mathrm{I}^h}\Delta x_i & x_i
	  			c_i^{n+1}(\varphi_i^{n+1}-\varphi_i^n)+\sum_{i=0}^{\mathrm{I}^h}\Delta x_i x_i
	  			c_i^{in}\varphi_i^0=\\ &\sum_{n=0}^{N-1}\sum_{i=0}^{\mathrm{I}^h}\int_{\tau_{n+1}}\int_{\Lambda_i^h}X^h(x)c^h(t,x)\frac{\varphi(t,\xi^h(x))-\varphi(t-\Delta
	  				t,\xi^h(x))}{\Delta t}dxdt \\&+\sum_{i=0}^{\mathrm{I}^h}\int_{\Lambda_i^h}X^h(x)c^h(0,x)\frac{1}{\Delta
	  				t}\int_0^{\Delta t}\varphi(t,\xi^h(x))dtdx,
	  		\end{align*}	
	  			which can be further rewritten as
	  			\begin{align}\label{25thdec_1}
	  				\sum_{n=0}^{N-1}\sum_{i=0}^{\mathrm{I}^h}\Delta x_i & x_i
	  				c_i^{n+1}(\varphi_i^{n+1}-\varphi_i^n)+\sum_{i=0}^{\mathrm{I}^h}\Delta x_i x_i
	  				c_i^{in}\varphi_i^0= \nonumber \\
	  				\int_{\Delta t}^T &\int_{0}^R X^h(x)c^h(t,x)\frac{\varphi(t,\xi^h(x))-\varphi(t-\Delta
	  					t,\xi^h(x))}{\Delta t}dxdt \nonumber \\&+\int_{0}^R X^h(x)c^h(0,x)\frac{1}{\Delta t}\int_0^{\Delta t}\varphi(t,\xi^h(x))dtdx.
	  			\end{align}
	  			
	  			Since, the derivative of $\varphi$ is bounded and $\varphi\in C^1([0,T[\times [0,R])$ is having compact support, thus  $$\frac{1}{\Delta t}\int_0^{\Delta
	  				t}\varphi(t,\xi^h(x))dt\rightarrow \varphi(0,x)\quad \text{as}\quad \max \{h,\Delta t\}\rightarrow 0$$ uniformly with respect to $t$ and $x$. Since, $c^h(0,x)\rightarrow c^{in}$ in $L^1]0,R]$, Lemma \ref{Wconverge} yields
	  			
	  			\begin{align*}
	  				\int_{0}^R X^h(x)c^h(0,x)\frac{1}{\Delta t}\int_0^{\Delta
	  					t}\varphi(t,\xi^h(x))dtdx\rightarrow \int_{0}^R x
	  				c^{in}(x)\varphi(0,x)dx
	  			\end{align*}
  			as $X^h(x)$ converges pointwise in $[0,R]$.	To treat the first part on the right-hand side of (\ref{25thdec_1}, a Taylor expansion of the smooth function $\varphi$ gives
	  			\begin{align*}
	  				\frac{\varphi(t,\xi^h(x))-\varphi(t-\Delta
	  					t,\xi^h(x))}{\Delta t}= & \frac{\varphi(t,x)+(x-\xi^h(x))\frac{\partial
	  						\varphi}{\partial x}-\varphi(t,x)+\Delta t \frac{\partial
	  						\varphi}{\partial t}-(x-\xi^h(x))\frac{\partial \varphi}{\partial
	  						x}+ O(h\,\Delta t)}{\Delta t},
	  			\end{align*}
	  			which means that, as max $\{h,\Delta t\}\rightarrow 0$,
	  			\begin{align*}
	  				\frac{\varphi(t,\xi^h(x))-\varphi(t-\Delta t,\xi^h(x))}{\Delta
	  					t}\rightarrow \frac{\partial \varphi}{\partial t}(t,x)
	  			\end{align*}
	  			uniformly. Again an application of Lemma \ref{Wconverge} and from Proposition \ref{equiintegrability}, one has
	  		
	  			\begin{align*}
	  				\int_0^T\int_0^R
	  				X^h(x)c^h(t,x)\frac{\varphi(t,\xi^h(x))-\varphi(t-\Delta
	  					t,\xi^h(x))}{\Delta t}dx\,dt \rightarrow \int_0^T\int_0^R x
	  				c(t,x)&\frac{\partial \varphi}{\partial t}(t,x)dx\,dt.
	  			\end{align*}
	  			Therefore,
	  			\begin{align*}
	  				\int_{\Delta t}^T \int_0^R
	  				\underbrace{X^h(x)c^h(t,x)\frac{\varphi(t,\xi^h(x))-\varphi(t-\Delta
	  						t,\xi^h(x))}{\Delta t}}_A dx\,dt&=\\ \int_0^T
	  				\int_0^R A\, dx\,dt -\int_0^{\Delta t} \int_0^R A\, dx\,dt & \rightarrow
	  				\int_0^T \int_0^R x c(t,x)\frac{\partial \varphi}{\partial
	  					t}(t,x)dx\,dt
	  			\end{align*}
	  			is established. At the end, taking the remaining terms of the equation (\ref{discrete24081}) and writing them in terms of $\mathcal{C}^h$ and $\mathcal{F}^h$ as
	  			
	\begin{align*}
	&	\sum_{n=0}^{N-1}\sum_{i=0}^{\mathrm{I}^h-1}\Delta t
	  	[\mathcal{C}_{i+1/2}^n-\mathcal{F}_{i+1/2}^n](\varphi_{i+1}^n-\varphi_i^n)- \sum_{n=0}^{N-1}\Delta t \mathcal{C}_{{\mathrm{I}}^h+1/2}^{n}\varphi_{{\mathrm{I}}^h}^n \\
	 &=\sum_{n=0}^{N-1}\sum_{i=0}^{\mathrm{I}^h-1}\int_{\tau_n}\int_{\Lambda_i^h}[\mathcal{C}_{i+1/2}^n-\mathcal{F}_{i+1/2}^n] \frac{1}{\Delta x_i}\left[\varphi(t,x_{i+1/2})-\varphi(t,x_{i-1/2})\right]dxdt\\
	 &- \sum_{n=0}^{N-1}\int_{\tau_n}\mathcal{C}_{{\mathrm{I}}^h+1/2}^{n}\varphi(t,R- \Delta x_{{\mathrm{I}}^h})dt\\
	 & = \int_0^T \int_0^{R-\Delta
	 	x_{\mathrm{I}^h}}[\mathcal{C}^h(t,x)-\mathcal{F}^h(t,x)] \frac{\partial \varphi}{\partial x}(t,x)dxdt -\int_{0}^{T} \mathcal{C}^h(t,R) \varphi(t,R-\Delta x_{{\mathrm{I}}^h})dt.	  				
	\end{align*}
 Thanks to Lemma {\ref{fluxconlemma}}, the weak convergence for the fluxes $\mathcal{C}^h\rightharpoonup \mathcal{C}_{nc}^R$ and $\mathcal{F}^h\rightharpoonup \mathcal{F}_{c}^R$ exists in $L^1(]0,T[\times ]0,R])$ which determine
	\begin{align*}
		&\int_0^T \int_0^{R-\Delta x_{\mathrm{I}^h}}[\mathcal{C}^h(t,x)-\mathcal{F}^h(t,x)] \frac{\partial \varphi}{\partial x}(t,x)dxdt -\int_{0}^{T} \mathcal{C}^h(t,R) \varphi(t,R-\Delta x_{\mathrm{I}^h})dt\\
		 & = \left(\int_0^T
		\int_0^R - \int_0^T \int_{\Delta x_{\mathrm{I}^h}}\right) [\mathcal{C}^h(t,x)-\mathcal{F}^h(t,x)] \frac{\partial \varphi}{\partial x}(t,x) dxdt- \int_{0}^{T} \mathcal{C}^h(t,R) \varphi(t,R-\Delta x_{\mathrm{I}^h})dt \\ & \rightarrow
		\int_0^T \int_0^R [\mathcal{C}_{nc}^R-\mathcal{F}_c^R] \frac{\partial \varphi}{\partial x}(t,x)dxdt-\int_{0}^{T} \mathcal{C}_{nc}^R(t,R) \varphi(t,R)dt
			\ \ \ \text{as}\ h\rightarrow 0.
	\end{align*}
	
	This complete the proof of Theorem 3.1 as all the terms in the equation (\ref{240800}) are obtained.
		\end{proof}

\section{Error Analysis}\label{errorde}
Here, we have discussed the error estimates for the coagulation and multiple fragmentation equations. It is important to mention here that, for the coagulation, results are taken from \cite{bourgade2008convergence}. So, our focus is to develop the study for multiple breakage model and combine the findings with the outcomes of \cite{bourgade2008convergence} for coagulation. Taking the uniform mesh is essential for estimating the error component, i.e., $\Delta x_i=h$\,\,$\forall i\in \{0,1,2,\ldots, \mathrm{I}^{h}\}$. The following theorem provides the first order error estimates by taking some assumptions about the kernels and initial datum.
	\begin{thm}\label{errorth1}
	Let the coagulation and fragmentation kernels satisfy $K,B\in W^{1,\infty}_{loc}(\mathbb{R}^{+} \times \mathbb{R}^{+})$ and selection rate, initial datum  $S,c^{in}\in W^{1,\infty}_{loc}(\mathbb{R}^{+})$.	Moreover, consider a  uniform volume mesh and  time step $\Delta t$  that satisfy the condition (\ref{22}). Then, the following error estimates
	\begin{align}\label{errorbound}
	\|c^h-c\|_{L^{\infty}(0,T;L^{1}(0, R))} \leq D(T, R)(h+\Delta t)
	\end{align}
	holds,	where $c$ is the weak solution to (\ref{maineq}) and $D(T, R)$ is a constant depending on $R$ and $T$. 
	\end{thm}
						
	Before proving the theorem, consider the following proposition, which gives an estimate on the approximate solution $c^h$ and the exact solution $c$ with certain additional assumptions. These computations are significant in predicting the error.
					
	\begin{propos}\label{bound2}
	Assume that kinetic parameters $K,B \in L^{\infty}_{loc}(\mathbb{R}^{+} \times \mathbb{R}^{+})$ and  $S \in  L^{\infty}_{loc}(\mathbb{R}^{+})$ and  the condition (\ref{22}) holds for time step $\Delta t$. Also, let the initial  datum $c^{in}$  restricted in $ L^{\infty}_{loc}$. Then, solution $c^h$ and $c$ to (\ref{maineq}) are essentially bounded in $(0,T)\times (0, R)$ as
	$$ \|c^h\|_{L^{\infty}((0,T)\times (0, R))}\leq D(T, R), \hspace{0.4cm} \|c\|_{L^{\infty}((0,T)\times (0, R))}\leq D(T, R). $$
	Furthermore, if the kernels $K, B \in    W^{1,\infty}_{loc}(\mathbb{R}^{+} \times \mathbb{R}^{+})$ and $S, c^{in} \in W^{1,\infty}_{loc}(\mathbb{R}^{+})$. Then there exists a positive constant $L(T, R)$ such that
	\begin{align}\label{bound1}
	\|c\|_{W^{1,\infty}(0, R)} \leq D(T, R).
	\end{align}
	\end{propos}
\begin{proof}
	The aim is to bound the solution $c$ to the continuous equation (\ref{maineq}). For this, integrating Eq. (\ref{main}) with respect to the time variable and leaving the negative terms out yield 
	\begin{align*}
	c(t,x) \leq c^{in}(x)+ \frac{1}{2}\int_{0}^{t}\int_0^x K(y,x-y)c(s,y)c(s,x-y)dy\,ds+\int_{0}^{t}\int_x^{R} B(x,y)S(y)c(s,y)dy\,ds.	
	\end{align*}						
	Thus, it follows
	\begin{align*}
	\sup_{x \in (\,0,\,R)\,}	c(t,x)\leq \underbrace{c^{in}(x)+ \|BS\|_{L^{\infty}} \|c\|_{\infty,1} t}_{\alpha(t)} + \underbrace{ {\|K\|_{L^{\infty}}} \|c\|_{\infty,1}}_{\beta} \int_{0}^{t}\sup_{y \in (\,0, R)\,} c(s,y)\,ds,
	\end{align*}
	where $\|c\|_{\infty,1}$ symbolizes the norm of $c$ in $L^{\infty}(0,T; L^{1}(\,0, R)\,)$. Subsequently, using  Gronwall's lemma and integration by parts lead to accomplish the proof as
	\begin{align*}
	\sup_{x \in (\,0, R)\,}	c(t,x) & \leq  \alpha(t)+ \int_{0}^{t}\alpha(s) \beta e^{\int_{s}^{t}\beta\,dr}\,ds\\
	&\leq\alpha(t)+\beta\Big[\frac{\alpha(s)e^{\beta(t-s)}}{-\beta}\Bigr|_{0}^{t} - \int_{0}^{t}\|SB\|_{L^{\infty}}\|c\|_{{\infty,1}}\frac{e^{\beta (t-s)}}{-\beta}\,ds \Big]\\	
	& \leq \alpha(0) e^{\beta t} + \frac{ \|BS\|_{L^{\infty}}  \|c\|_{\infty,1}}{\beta}[(e^{\beta t}-1) ].
	\end{align*}
	Therefore, 
	\begin{align*}
	\|c\|_{L^{\infty}((0,T) \times (\,0, R)\,)} \leq D(T, R).
	\end{align*}
	Now, moving to the conclusion of an estimate of (\ref{bound1}). To begin with, integrate the Eq. (\ref{main}) with respect to time variable $t$ and next, differentiate it with respect to the volume variable $x$. The maximum value over the domain of $x$ is then achieved
	\begin{align*}
   \left\|\frac{\partial c}{\partial x}(x)\right\|_{L^{\infty}} \leq &	\left\|\frac{\partial c^{in}}{\partial x}\right\|_{L^{\infty}}+ \Big\{\frac{1}{2}  \|K\|_{L^{\infty}}  \|c\|_{L^{\infty}}^{2} 
	+ 2  \|K\|_{W^{1,\infty}} \|c\|_{L^{\infty}} \|c\|_{\infty,1}  R+ \|BS\|_{L^{\infty}}  \|c\|_{L^{\infty}} \\
	& + 2 \|B\|_{W^{1,\infty}}  \|S\|_{L^{\infty}}  \|c\|_{\infty,1}  R+ \|S\|_{W^{1,\infty}}  \|c\|_{L^{\infty}} \Big \} t\\
	& + \left(2  \|K\|_{L^{\infty}}  \|c\|_{\infty,1} R + \|S\|_{L^{\infty}}\right) \int_{0}^{t}\left \|\frac{\partial c}{\partial x}\right\|_{L^{\infty}}\,ds.
	\end{align*}
	Applying Gronwall's lemma as used in priori boundedness of $c$ to establish (\ref{bound1}). 
	\end{proof}	
	The discrete coagulation and fragmentation terms in (\ref{gun}) are written for uniform mesh as follows:
	\begin{align}\label{discrte1}
	-\frac{\mathcal{C}_{i+1/2}^n-\mathcal{C}_{i-1/2}^n}{h}=h\sum_{j=0}^{i-1}x_j K_{j,i-j-1}c_{i}^{n}c_{i-j-1}^{n}-h\sum_{j=0}^{\mathrm{I}^h}x_i K_{i,j}c_{i}^{n}c_{j}^{n},
	\end{align}
	and
	\begin{align}\label{discrte2}
	\frac{\mathcal{F}_{i+1/2}^n-\mathcal{F}_{i-1/2}^n}{h}= -\sum_{j=0}^{i-1}x_j S_{i} B_{j,i}c_{i}^{n}\Delta x_{j}+\sum_{j=i+1}^{\mathrm{I}^h}x_i S_{j} B_{i,j}c_{j}^{n}\Delta x_{j}.
	\end{align}
	The discrete terms above are transformed into a continuous expression using the following lemma.
	\begin{lem}
	Consider the initial condition $c^{in}$ $\in W^{1,\infty}_{loc}$ and uniform mesh, $\Delta x_{i}=h$ $\forall i$. Also assuming that $K$, $B$, and $S$ follow the conditions   $K,B,S\in W^{1,\infty}_{loc}.$ Let $(s,x)\in \tau_{n}\times \Lambda_{i}^{h}$, where $n\in \{0,1,\ldots, N-1\}\, , i\in\{0,1,2,\ldots,\mathrm{I}(h)\}$. Then
	\begin{align}\label{convert1}
	-\frac{\mathcal{C}_{i+1/2}^n-\mathcal{C}_{i-1/2}^n}{x_{i}h}=&\frac{1}{2}\int_{0}^{\xi^h(x)}K^h(x^{'},x-\Xi^{h}(x^{'}))c^{h}(s,x^{'})c^{h}(s,x-\Xi^{h}(x^{'}))\,dx{'}\\ \nonumber
	&-\int_{0}^{R}K^h(x,x^{'})c^{h}(s,x)c^{h}(s,x^{'})\,dx^{'}+\varepsilon(\mathbb{C},h),
	\end{align}
							
	\begin{align}\label{convert2}
\frac{\mathcal{F}_{i+1/2}^n-\mathcal{F}_{i-1/2}^n}{x_{i}h}=\int_{\Xi^{h}(x)}^{R}S^{h}(x^{'})B^{h}(x,x^{'})c^{h}(s,x^{'})\,dx^{'}-S^{h}(x)c^{h}(s,x)+\varepsilon(\mathbb{F},h),
	\end{align}
	where $\varepsilon(\mathbb{C},h)$ and  $\varepsilon(\mathbb{F},h)$ expresses the first order term with respect to $h$ in the strong $L^{1}$ topology:
	\begin{align}
	\|\varepsilon(\mathbb{C},h)\|_{L^1}\leq \frac{R}{2}\|c^{h}\|^{2}_{L^{\infty}}\|K\|_{L^{\infty}}h,
	\end{align}
	\begin{align}
	\|\varepsilon(\mathbb{F},h)\|_{L^1}\leq {R}\|BS\|_{L^{\infty}}\|c^{h}\|_{\infty,1} h.
	\end{align}
	\end{lem}
	\begin{proof}
	The simplified version of coagulation part (\ref{convert1}) from (\ref{discrte1}) has been investigated in \cite{bourgade2008convergence}.
	Here, we would discuss only multiple fragmentation variation rate (\ref{discrte2}) to convert into  Eq. (\ref{convert2}) using a
	uniform mesh and having $x \in \Lambda_{i}^{h}$. Consider	
	\begin{align*}
	\frac{\mathcal{F}_{i+1/2}^n-\mathcal{F}_{i-1/2}^n}{x_{i}h}=& \sum_{j=i+1}^{\mathrm{I}^h}x_i S_{j} B_{i,j}c_{j}^{n}\Delta x_{j}-\frac{S_{i}c_{i}^{n}}{x_{i}}\sum_{j=0}^{i-1}x_j  B_{j,i}\Delta x_{j}\\ 
	= &\sum_{j=i+1}^{\mathrm{I}^h}x_i S_{j} B_{i,j}c_{j}^{n}\Delta x_{j}-\frac{S_{i}c_{i}^{n}}{x_{i}}\sum_{j=0}^{i}x_j  B_{j,i}\Delta x_{j}+S_{i}c_{i}^{n}B_{i,i}\Delta x_{i}\\
	=&\int_{\Xi^{h}(x)}^{R}S^{h}(x^{'})B^{h}(x,x^{'})c^{h}(s,x^{'})\,dx^{'}-S^{h}(x)c^{h}(s,x)+\varepsilon(\mathbb{F},h).
	\end{align*}
	Let $\varepsilon(\mathbb{F},h)= S_{i}c_{i}^{n}B_{i,i}h$ is considered, and when the $L_{1}$ norm of $\varepsilon(\mathbb{F},h)$ is calculated, the following term emerges
\begin{align*}
\|\varepsilon(\mathbb{F},h)\|_{L^1}&\leq\|BS\|_{L^{\infty}}h\sum_{i=1}^{\mathrm{I(h)}}\int_{\Lambda_{i}^{h}}c_{i}^{n}\,dx\\
& \leq \|BS\|_{L^{\infty}} \|c^{h}\|_{\infty,1}{h}\sum_{i=1}^{\mathrm{I(h)}}\int_{\Lambda_{i}^{h}}\,dx\\
	& \leq R\|BS\|_{L^{\infty}} \|c^{h}\|_{\infty,1} h.
\end{align*}
Now, to demonstrate the essential fact, Theorem 5.1, when combined with Eqs. (\ref{gun}), (\ref{convert1}) and (\ref{convert2}) yields
	\begin{align}\label{maineq2}
						\frac{\partial c^{h}(t,x)}{\partial t}	=&\frac{1}{2}\int_{0}^{\xi^h(x)}K^h(x^{'},x-\Xi^{h}(x^{'}))c^{h}(s,x^{'})c^{h}(s,x-\Xi^{h}(x^{'}))\,dx{'}-\int_{0}^{R}K^h(x,x^{'})c^{h}(s,x)c^{h}(s,x^{'})\,dx^{'} \nonumber \\
						& + \int_{\Xi^{h}(x)}^{R}S^{h}(x^{'})B^{h}(x,x^{'})c^{h}(s,x^{'})\,dx^{'}-S^{h}(x)c^{h}(s,x)+\varepsilon(\mathbb{C},h)+\varepsilon(\mathbb{F},h).
					\end{align}
			Finally, we may derive the error formulation for $t\in \tau_n$ from Eqs. (\ref{main}) and (\ref{maineq2}) as
	\begin{align}\label{errorfull}
	\int_{0}^{R}|c^h(t,x)-c(t,x)|dx \leq  \int_{0}^{R}|c^h(0,x)-c(0,x)|dx + \sum_{\beta=1}^{4}\epsilon_{\beta}(\mathbb{C},h) +\sum_{\beta=1}^{3}\epsilon_{\beta}(\mathbb{F},h)\nonumber \\ 
	+\int_{0}^{R}|\epsilon(t,n)|\,dx+\|\varepsilon(\mathbb{C},h)\|_{L^{1}} t+ \|\varepsilon(\mathbb{F},h)\|_{L^{1}} t,
	\end{align}
where error terms are denoted by $\epsilon_{\beta}(\mathbb{C}, h)$ with $\beta$= 1,2,3,4 in relation to the coagulation operator and estimation of these terms are calculated in \cite{bourgade2008convergence}. In this article, only $\epsilon_{\beta}(\mathbb{F}, h)$ with $\beta$= 1,2,3 for fragmentation operator and $\int_{0}^{R}|\epsilon(t,n)|\,dx$ would be estimated. Here,
	\begin{align*}
	\epsilon_{1}(\mathbb{F},h)= \int_{0}^{t}\int_{0}^{R}\int_{\Xi^{h}(x)}^{R}|S^{h}(x^{'})B^{h}(x,x^{'})c^{h}(s,x^{'})-S(x^{'})B(x,x^{'})c(s,x^{'})|\,dx^{'}\,dx\,ds,
	\end{align*}
\begin{align*}
\epsilon_{2}(\mathbb{F},h)= \int_{0}^{t}\int_{0}^{R}\int_{x}^{\Xi^{h}(x)}S(x^{'})B(x,x^{'})c(s,x^{'})\,dx^{'}\,dx\,ds,	
\end{align*}
and
\begin{align*}
\epsilon_{3}(\mathbb{F},h)= \int_{0}^{t}\int_{0}^{R}|S^{h}(x)c^{h}(s,x)-S(x)c(s,x)|\,dx\,ds.
\end{align*}
Furthermore, due to time discretization, assuming $|t-t_n|\leq \Delta t$ produces
\begin{align*}
\int_{0}^{R}|\epsilon(t,n)|\,dx\leq & \frac{1}{2}\int_{t_n}^{t}\int_{0}^{R}\int_{0}^{\xi^{h}(x)}K^h(x^{'},x-x^{'})c^{h}(s,x^{'})c^{h}(s,x-x^{'})\,dx^{'}\,dx\,ds\\
	& + \int_{t_n}^{t}\int_{0}^{R}\int_{0}^{R}K^h(x,x^{'})c^{h}(s,x)c^{h}(s,x^{'})\,dx^{'}\,dx\,ds\\
	& + \int_{t_n}^{t}\int_{0}^{R}\int_{\Xi^{h}(x)}^{R}S^{h}(x^{'})B^{h}(x,x^{'})c^{h}(s,x^{'})\,dx^{'}\,dx\,ds\\
	& + \int_{t_n}^{t}\int_{0}^{R}S^{h}(x)c^{h}(s,x)\,dx\,ds+ \int_{t_n}^{t}\int_{0}^{R} \varepsilon(\mathbb{C},h)\,dx\,ds+ \int_{t_n}^{t}\int_{0}^{R} \varepsilon(\mathbb{F},h)\,dx\,ds.
	\end{align*}
	
Utilizing the smoothness property of kernels, i.e., $K, B \, \text{and} \,\,S\in W_{loc}^{1,\infty}$, we have for all $x,y\in (0,R)$ 		$$|B^{h}(x,y)-B(x,y)|\leq \|K\|_{W^{1,\infty}} h. $$
Thus, it provides the estimation of $\epsilon_{1}(\mathbb{F},h)$  using the $L^{\infty}$ bound on $c^h$ and $c$. To see this, we split the expression into three parts as
		\begin{align*}
	\epsilon_{1}(\mathbb{F},h)\leq & \int_{0}^{t}\int_{0}^{R}\int_{0}^{R}|S^h(x^{'})-S(x^{'})|B(x,x^{'})c(s,x^{'})\,dx^{'}\,dx\,ds	\\
	& + \int_{0}^{t}\int_{0}^{R}\int_{0}^{R}S^h(x^{'})|B^h(x,x^{'})-B(x,x^{'})|c(s,x^{'})\,dx^{'}\,dx\,ds  \\
	&+\int_{0}^{t}\int_{0}^{R}\int_{0}^{R}S^h(x^{'})B^h(x,x^{'})|c^{h}(s,x^{'})-c(s,x^{'})\,dx^{'}\,dx\,ds  .
	\end{align*}					
	The above may be transformed to, by simplifying and using Proposition 5.2
		\begin{align}\label{frag1}
		\epsilon_{1}(\mathbb{F},h) \leq t  R^{2}  \|c\|_{\infty}  (\|B\|_{\infty}\|S\|_{W^{1,\infty}} + \|S\|_{\infty}\|B\|_{W^{1,\infty}}) h	+ R  \|BS\|_{\infty}\int_{0}^{t}\|c^{h}(s)-c(s)\|_{L^1}\,ds.
	\end{align}
	Similarly, one can compute
	\begin{align}\label{frag3}
	\epsilon_{3}(\mathbb{F},h) \leq \|c\|_{\infty} \|S\|_{W^{1,\infty}} t R h+\|S\|_{\infty} \int_{0}^{t}\|c^{h}(s)-c(s)\|_{L^1}\,ds,
	\end{align}
	and 
	\begin{align}\label{frag2}
							\epsilon_{2}(\mathbb{F},h) \leq \frac{t R}{2}\|BS\|_{\infty}\|c\|_{\infty} h.
						\end{align}
Now, let us move on to the remaining term $\int_{0}^{R}|\epsilon(t,n)|\,dx$, the error introduced by time discretization is resolved, and a bound is found as
		\begin{align}\label{time1}
		\int_{0}^{R}|\epsilon(t,n)|\,dx\leq \Big(\frac{3}{2}\|K\|_{\infty}\|c^h\|_{\infty}^{2}R^2+ (\|B\|_{\infty}R+1)\|S\|_{\infty}\|c^h\|_{\infty}R +\|\varepsilon(\mathbb{C},h)\|_{L^{1}} +  \|\varepsilon(\mathbb{F},h)\|_{L^{1}}\Big) \Delta t.
		\end{align}
				In conclusion, assemble all the bound estimations on $\epsilon_{\beta}(\mathbb{C}, h)$ for $\beta$=1,2,3,4  from \cite{bourgade2008convergence}, $\epsilon_{\beta}(\mathbb{F}, h)$ for $\beta$=1,2,3 from Eq. (\ref{frag1})-(\ref{frag2}) and the relation Eq. (\ref{time1}). Substituting all these estimations in (\ref{errorfull}) and applying the Gronwall's lemma conclude the proof as 
\begin{align*}\label{errorbound}
\|c^h-c\|_{L^{\infty}(0,T;L^{1}(0, R))} \leq D(T, R)(h+\Delta t).
\end{align*}						
\end{proof}
\section{Numerical Results}\label{errornu}
In this section, we numerically validate the conclusions obtained in Section \ref{errorde} for two test problems. As a result, the goal is solely to demonstrate experimental error and experimental order of convergence (EOC) for C-F equation having different kernels on uniform mesh. As, the analytical solutions are not available for such cases due to singularity in breakage kernel, the following expression is used to calculate the EOC
	\begin{align}
	EOC=\ln \left( \frac{\|N^{num}_{\mathrm{I}^h}-N^{num}_{2\mathrm{I}^h}\|}{\|N^{num}_{2\mathrm{I}^h}-N^{num}_{4\mathrm{I}^h}\|} \right)/\ln(2).
	\end{align}
Here, the total number of particles created by the FVS (\ref{gun}) with a mesh of $\mathrm{I}^h$ number of cells is denoted by $N^{num}_{\mathrm{I}^h}$. The computational domain of volume and time arguments in simulations are taken as $x =[1e-3, 100]$ and $t=100$ with $e^{-x}$ as the initial condition.
	\subsection{Test Case 1 }
	Consider the sum coagulation kernel, i.e.,  $K(x,y)= x+y$ with selection function $S(x)=x^{1/2},x^{1/4}$ and breakage function $B(x,y)= \frac{\alpha+2}{y}\bigg(\frac{x}{y}\bigg)^{\alpha}$, $\alpha =-1/2$. The EOC for uniform mesh is shown in Table 1. As expected from theoretical results, it is clear from the table that the FVS produces first-order convergence. The numerical errors are evaluated using 30, 60, 120, 240 and 480 degrees of freedom.
	
	\begin{table}[htb]
	 \begin{minipage}{.5\linewidth}
	  \centering
	    \begin{tabular}{ |p{1cm}|p{1.5cm}|p{1cm}|}
	    \hline
	 \multicolumn{3}{|c|}{$S(x)=x^{1/2}$} \\
	    \hline
	      Cells      & Error & EOC\\
	       \hline
	       30    & - & -  \\
	          \hline
	      60& 0.1499  & - \\
	           \hline
	       120  &   0.0272   &0.9743   \\
	          \hline
	        240  & 0.0039 & 0.9953  \\
	          \hline
	      480  & 0.0005 & 0.9992   \\
	 \hline
	           \end{tabular}
	           \end{minipage}
	           \begin{minipage}{.5\linewidth}
	           \centering
	            \begin{tabular}{ |p{1cm}|p{1.5cm}|p{1cm}|}
	                                          \hline
	         \multicolumn{3}{|c|}{$S(x)=x^{1/4}$} \\
	            \hline
	            Cells      & Error & EOC\\
	            \hline
	              30    & - &-   \\
	                 \hline
	              60& 0.0014  &- \\
	             \hline
	              120   & 0.0005      & 0.9997 \\
	                 \hline
	               240  & 0.0002 & 1.0001  \\
	               \hline
	               480  & 0.0001 & 1.0000    \\
	            \hline
	            \end{tabular}
	              \end{minipage}
	           \caption{EOC for Test Case 1}
	           \end{table}  
	            \subsection{Test Case 2 }
Now, we consider the same coagulation kernel and the selection functions taken in the previous test case but the breakage function $B(x,y)$ is used with $\alpha=-3/4$. The error and EOC of the scheme are reported in Table 2 by computing the errors for 30, 60, 120, 240 and 480 number of grid points, and it is observed again that the FVS is first order accurate. It should also be emphasized that the EOC is evaluated for a variety of distinct $\alpha$, $K$, and $S$. However, identical outcomes are achieved in every case, thus, the results are excluded here.
	            \begin{table}[htb]
	         \begin{minipage}{.5\linewidth}
	             \centering
	          \begin{tabular}{ |p{1cm}|p{1.5cm}|p{1cm}|}
	                \hline
	           \multicolumn{3}{|c|}{$S(x)=x^{1/2}$} \\
	             \hline
	               Cells      & Error & EOC\\
	                 \hline
	                30    & - &-   \\
	                   \hline
	              60&  0.1179 &- \\
	                 \hline
	               120  & 0.0195                & 0.9784 \\
	                    \hline
	                240  & 0.0026 & 0.9965  \\
	                \hline
	                  480  & 0.0003 &  0.9995  \\
	             \hline
	          \end{tabular}
	          \end{minipage}
	           \begin{minipage}{.5\linewidth}
	        \centering
	      \begin{tabular}{ |p{1cm}|p{1.5cm}|p{1cm}|}
	                         \hline
	      \multicolumn{3}{|c|}{$S(x)=x^{1/4}$} \\
	                 \hline
	                 Cells      & Error & EOC\\
	                         \hline
	                              30    &-  & -  \\
	                             \hline
	                   60& 0.0830  &- \\
	                    \hline
	                 120  &  0.0371  &0.9801  \\
	                                                         \hline
	                    240  & 0.0062 & 0.9930  \\
	                           \hline
	                  480  &0.0008  & 0.9988   \\
	                        \hline
	                                  \end{tabular}
	                                  \end{minipage}
	                       \caption{EOC for Test Case 2}
	                       \end{table}    
\section{Conclusions}
This work dealt with the convergence of finite volume truncated solutions towards a weak solution to the continuous coagulation and multiple fragmentation equations under the assumptions that coagulation kernel was locally bounded whereas breakage has singularity near the origin. The weak $L^1$ compactness argument was used to establish the result. Further, for more restricted classes of kernels, the scheme was shown to provide the first order error estimates on uniform meshes. This finding was verified numerically by taking different examples of coagulation and breakage kernels.  
	
\section{Acknowledgments}
The research described in this paper evolved as part of the research project (File Number: SRG/2019/001490) funded by the Science and Engineering Research Board, India.	
	
	\bibliography{reference}
	\bibliographystyle{ieeetr}

\end{document}